\documentclass[12pt]{amsart}

\usepackage{amsmath}
\usepackage{amssymb}
\usepackage{amsthm}
\usepackage{amsfonts}
\usepackage[all]{xy}
\usepackage[margin=3cm]{geometry}
\usepackage{hyperref}

\newcommand{\CC}{\mathbb{C}}

\newcommand{\NN}{\mathbb{N}}

\newcommand{\Aff}{\mathbb{A}}
\newcommand{\PP}{\mathbb{P}}

\newcommand{\Spec}[1]{\text{Spec\;}#1}
\newcommand{\tensor}{\otimes}

\newcommand{\OO}{\mathcal{O}}

\newcommand{\und}[1]{\underline{#1}}

\newcommand{\Kim}{\mathcal{M}_{\Gamma}^{\textup{log}}(\mathcal{U/B})}
\newcommand{\Li}{\mathcal{M}_{\Gamma}(X,D)}
\newcommand{\Kims}[1]{\mathcal{M}_{#1}^{\textup{log}}(\mathcal{U/B})^{\textup{sim}}}
\newcommand{\Kimc}[1]{\mathcal{M}_{#1}^{\textup{log}}(\mathcal{U/B})^{\sim}}

\theoremstyle{definition}
\newtheorem{definition}{Definition}

\newtheorem{remark}{Remark}

\theoremstyle{plain}
\newtheorem{thm}{Theorem}
\newtheorem{lem}{Lemma}
\newtheorem{cor}{Corollary}

\title{Localization for Logarithmic Stable Maps}

\date{}

\author{S. Molcho}
\address{Department of Mathematics
Campus Box 395 
Boulder, Colorado 80309-0395}
\email{Samouil.Molcho@colorado.edu}
\author{E. Routis}
\address{Brown University, Department of Mathematics,
Box 1917
151 Thayer Street
Providence, RI 02912}
\email{r0utis@math.brown.edu}
\title{Localization for Logarithmic Stable Maps}

\begin{document}
\maketitle
\bibliographystyle{amsalpha}

\begin{abstract}
We prove a virtual localization formula for Bumsig Kim's space of Logarithmic Stable Maps. The formula is closely related and can in fact recover the relative virtual localization formula of Graber-Vakil. 
\end{abstract}

\begin{section}{Introduction and Background}
In his paper \cite{Li}, Jun Li introduced and studied the space of relative stable maps. We recall the setup: fix a pair $(X,D)$ of a smooth variety with a smooth divisor, and discrete data $\Gamma = (g, \vec{c} = (c_1,\cdots, c_h), \beta)$, consisting of the arithmetic genus $g$ of a nodal curve, a vector $\vec{c}$ of integers, and a homology class $\beta$ in $X$. We wish to parametrize stable maps $f:(C,\vec{y},\vec{x}) \rightarrow X$ from a genus $g$ nodal curve $C$ with two sets of marked points $\vec{y}=(y_1,\cdots,y_m)$ and $\vec{x}=(x_1,\cdots,x_h)$ into the variety $X$, whose image lies in the given homology class and with prescribed incidence conditions along the divisor, namely, $f^{-1}(D)= \sum c_ix_i$. \\

The moduli space parametrizing is not proper: a limit of such maps may fail to exist, as in the limit the whole curve may lie entirely into the divisor $D$. Jun Li, following ideas of Li-Ruan \cite{LR} and Ionel-Parker \cite{IP} from symplectic Gromov-Witten theory, gave the following beautiful solution to this issue. When a limit of maps tends to collapse into the divisor, the space $X$ sprouts a new component, which is isomorphic to the projective completion of the normal bundle $\PP(N_{D/X} \oplus 1)$ of $D$ to contain the image, in a manner similar to a blowup. We then require that the prescribed behavior along the divisor does not happen along the original divisor $D$, but rather the divisor at infinity in $\PP(N_{D/X} \oplus 1)$, which we denote $D[1]$. We call $X$ with this new component $X[1]$; we then have a new pair $(X[1],D[1])$ and we may consider stable maps as above to this pair. When a family of maps to $X[1]$ ten
 ds to collapse into $D[1]$, the variety $X[1]$ sprouts a new component that replaces $D[1]$, as above, to create a new space $X[2]$ with a divisor $D[2]$ at infinity, and so forth. In general, a pair $(X[n],D[n])$ is constructed from the pair $(X[n-1],D[n-1])$ inductively. It is called the $n$-th expansion of $(X,D)$. Li's moduli space $\Li$ parametrizes stable maps whose target is allowed to be any of the expansions $(X[n],D[n])$ above, with prescribed behavior along the divisor at infinity $D[n]$, and with a certain compatibility requirement along the divisor $D[k]$, $k<n$: only nodes of the source curve can map to $D[k]$, and when a node maps to $D[k]$, the two components of the curve containing the node have the same contact order with the divisor $D[k]$; this is called the predeformability or ''kissing'' condition. This space is proper and is shown to carry a virtual fundamental class, so one can define in a standard manner a type of Gromov-Witten invariants for $(X,D)$, called \und{relative} Gromov-Witten invariants. For details of the construction, the reader should consult Jun Li's original paper \cite{Li}. \\

Jun Li also considers a variant of this situation, where instead of a pair $(X,D)$ we consider a semistable nodal variety of the form $X=Y_1 \cup_D Y_2$. This means that $X$ is the union of two smooth varieties $Y_i$ along a common smooth divisor $D$ in both of them, that satisfies the following technical condition on the normal bundles: $N_{D/Y_1} \tensor N_{D/Y_2} \cong 1$. Stable maps into $X$ must satisfy a similar predeformability condition as above, and the space is compactified by allowing the targets to vary as before. $X$ maydeform to a target $X[1]$, where $D$ is replaced by $\PP(N_{D/Y_1} \oplus 1) \cong \PP(N_{D/Y_2} \oplus 1)$, with $Y_1$ glued along the $0$ section and $Y_2$ along the infinity section, $X[1]$ may deform to $X[2]$ where the divisor connection $Y_2$ with $\PP(N_{D/X} \oplus 1)$ is replaced by another copy of $\PP(N_{D/X} \oplus 1)$, and so forth. The space $X[n]$ are called the expanded degenerations of $X$. The space of expanded degenerations also carries a virtual fundamental class and one is thus able to extend the notion of Gromov-Witten invariants for targets $X=Y_1 \cup Y_2$, which are mildly singular. These are the correct Gromov-Witten invariants, in the sense that they satisfy deformation invarance: If $W \rightarrow B$ is a family with smooth total space, smooth general fiber and central fiber $X = Y_1 \cup_D Y_2$, the Gromov-Witten invariants of $X$ as defined by Jun Li coincide with the usual Gromov-Witten invariants of the general fiber, at least when such a comparison makes sense, i.e for homology classes restricted from $W$. \\

The relative Gromov-Witten invariants are related to the singular Gromov-Witten invariants by the \und{degeneration} formula. This was also proven by Jun Li and had also been previously considered in the symplectic category in the work of Li-Ruan \cite{LR} and Ionel-Parker \cite{IP}. The degeneration formula allows one to compute Gromov-Witten invariants of expanded degenerations from the relative ones and the combinatorics of the expansions. This can be useful because it is often possible to degenerate a smooth variety into a semistable one with very simple components $Y_i$. Thus one can calculate Gromov-Witten invariants from relative Gromov-Witten invariants of simpler targets. \\

Computations of relative Gromov-Witten invariants can still be difficult, as calculations in Gromov-Witten theory often are, even if the targets are very simple. These calculations can be greatly facilitated by the use of Atiyah-Bott localization. Localization fromulas for the spaces $\Li$ were established by Graber-Vakil in \cite{GV}. The applications of such formulas are far reaching: for example, in \cite{GV}, as applications of the formulas the authors recover the ELSV formula and certain striking results about the tautological ring. \\

Jun Li's constructions are beautiful and geometrically transparent, but suffer from one technical drawback. The virtual fundamental classes defined are hard to work with. The reason for this is that the space of relative stable maps is not an open subset of all maps, but rather, it is locally closed. The perfect obstruction theory used to define the virtual fundamental class is thus constructed by hand and not by standard machinery. This is the main reason the paper \cite{GV} is technically difficult. \\

One way to avoid this issue is to use a different compactification of the space of maps to the pair $(X,D)$ or $Y_1 \cup_D Y_2$, by endowing the sources and targets of all maps with logaritmic structures and requiring that the maps between them are log maps. We will explain this more precisely in the sequel, but here we would like to remark that this idea agrees with a general philosophy in the modern theory of moduli that states that instead of compactifying a moduli space of certain objects, on may try to build the moduli space of such objects with logarithmic structures; since logarithmic structures allow mild singularities, this moduli space is often already proper. The space of logarithmic stable maps was constructed by B.Kim in his paper \cite{Kim}. Kim's space is shown to be an open substack of the space of all logarithmic maps and thus carries a natural virtual fundamental class by restriction, which is simpler than the virtual fundamental class of $\Li$: its formal properties are almost identical to the fundamental class in the classical Gromov-Witten theory of smooth targets. The situation may be summarized pictorially as follows:

\begin{align}
\xymatrix{\ar @{} [dr] |{}
\textrm{log stable maps} \ar[d]_{\textrm{forgetful}} \ar[r]^{\textrm{open}} & \textrm{log maps} \ar[d]^{\textrm{forgetful}} \\
\textrm{rel stable maps} \ar[r]^{\textrm{loc closed}} & \textrm{all maps} }
\tag{1} \label{eq:1}
\end{align}

In this paper, we derive an analogous localization formula for the space $\Kim$ of logarithmic stable maps; the formula is analogous to the formula of \cite{GV} but its derivation is closer in spirit with the proofs of localization in classical Gromov-Witten theory, as in \cite{KTorus},\cite{GP}. Specifically, we show

\newtheorem*{thm:1}{Theorem \ref{thm:1}}
\begin{thm:1}
$\Kim$ is a global quotient stack and admits a $\mathbb{C}^*$ equivariant immersion to a smooth Deligne-Mumford stack. 
\end{thm:1}

\noindent This in particular shows that $\Kim$ admits a localization formula. Then, following the work of Graber-Vakil \cite{GV}, we obtain explicitly that for suitable splittings of the discrete data $\Gamma$ into subsets $\Gamma_1,\Gamma_2$, we have 

\newtheorem*{thm:2}{Theorem \ref{thm:2}}
\begin{thm:2}{Log Virtual Localization}:
\begin{align*}
[\Kim]^{\textup{vir}} = [\Kims{\Gamma}]^{\textup{vir}} + \sum_{\Gamma_1,\Gamma_2} \frac{\prod{\alpha_i}}{d}\frac{[\Kims{\Gamma_1} \times_{D^k} \Kimc{\Gamma_2}]^{\textup{vir}}}{Aut(\Gamma_1,\Gamma_2)(\frac{w-\psi}{d})e(N_{\Gamma_1})}
\end{align*}
\end{thm:2}

\noindent The formula is essentially the same as the relative virtual localization formula of \cite{GV}. The difference is that the stacks of simple relative maps and unrigidified relative stable maps of \cite{GV} are replaced by their logarithmic analogues. These are defined carefully in section 5. In fact, the log virtual localization formula can be used to recover the formula of \cite{GV} - this is our corollary 7. \\

\end{section}

\begin{section}{Logarithmic Stable Maps}

For completeness, we will recall here the necessary definitions and constructions that we will use. For proofs and more detailed explanations the reader should consult Kim's paper \cite{Kim}. \\

A family of $n$-marked prestable curves, $C \rightarrow S$ carries a canonical structure of a logarithmic map, as shown in F. Kato's paper \cite{fK}. The log structures and morphisms are defined as follows. The curve $C \rightarrow S$ corresponds to a diagram

\begin{align*}
\xymatrix{\ar @{} [dr] |{}
C \ar[d] \ar[r] & \mathfrak{C}_{g,n} \ar[d] \\
S \ar[r] & \mathfrak{M}_{g,n} }
\end{align*}
\noindent where $\mathfrak{M}_{g,n}$ and $\mathfrak{C}_{g,n}$ is the moduli stack of $n$-marked prestable curves and its universal family respectively. Both stacks carry natural logarithmic structures: in $\mathfrak{M}_{g,n}$ the log structure is given by the divisor corrseponding to singular curves, and in $\mathfrak{C}_{g,n}$ the log structure is given by the divisors corresponding to singular curves and the markings. The log structures on $C$ and $S$ are the ones pulled back from $\mathfrak{C}_{g,n}$ and $\mathfrak{M}_{g,n}$ respectively; we denote the log structure on $S$ by $N^{C/S}$ and on $C$ by $M^{C/S}$ and refer to them as the canonical log structures. The morphism $C \rightarrow S$ is automatically a log morphism. An explicit description of the log structures when $S = \Spec {k}$ is a geometric point can be given in terms of charts as follows: $N^{C/S}$ has a chart isomorphic to $\NN^{m}$, $m = \textrm{the number of nodes of } C$; $M^{C/S} = N^{C/S}$ at smooth points; $M^{C/S}= N^{C/S} \oplus \NN$ at a marked point; and at a node $M^{C/S}$ is given by the following pushout diagram:

\begin{align*}
\xymatrix{\ar @{} [dr] |{}
\NN \ar[d] \ar[r] & \NN^2 \ar[d] \\
N^{C/S} \ar[r] & M^{C/S} }
\end{align*}

\noindent where the horizontal map is the diagonal and the vertical map is the inclusion corresponding to the appropriate node. \\

\begin{definition} A genus $g$, $n$-marked log curve is a morphism $f:(C,M) \rightarrow (S,N)$ of log schemes such that $C \rightarrow S$ is a family of genus $g$, $n$-marked prestable curves and the morphism $f$ is obtained from a cartesian diagram of the form

\begin{align*}
\xymatrix{\ar @{} [dr] |{}
(C,M) \ar[d] \ar[r] & (C,M^{C/S}) \ar[d] \\
(S,N) \ar[r] & (S,N^{C/S}) }
\end{align*}

\noindent where the horizontal maps are the identities on underlying schemes. Therefore, a log curve is the same thing as the choice of a prestable curve $C \rightarrow S$ and the choice of a homomorphism of log structures $s^{C/S}: N^{C/S} \rightarrow N$. We will reserve the notation $s^{C/S}$ to always refer to this homomorphism, and denote it by $s$ to simplify notation if no confusion may arise. 

\end{definition}

We denote the stack parametrizing log curves by $\mathfrak{M}_{g,n}^{\textup{log}}$. \\

\begin{definition} A log curve $(M,C) \rightarrow (S,N)$ is called \emph{minimal} if the log structure $N$ is locally free and there is no locally free submonoid $N' \subset N$ that contains the image of $N^{C/S}$.
\end{definition}

Here, we call a log structure locally free if around every point it has a chart isomorphic to $\NN^{r}$ for some $r$, possibly depending on the point. For example, over $S = \Spec k$, where $N^{C/S}=\NN^{m}$, all surjections $\NN^{m} \rightarrow \NN^{r}$, $r \le m$ give minimal log curves but no map $\NN^{m} \rightarrow \NN^{r}$, $r>m$.  \\

Minimal log curves are essentially the sources of log stable maps - for a more precise statement, see Definition 5. Next, we discuss the possible targets, which Kim calls ``twisted expanded log Fulton-Macpherson type spaces''. We will refer to them more concisely as log FM spaces. \\

Fix a smooth projective variety $X$.

\begin{definition}
A family of schemes or algebraic spaces $W \rightarrow S$ is called a log FM type space of $X$ if, at every point $s \in S$, there is \'{e}tale locally an \'{e}tale map 

\begin{align*}
W_{\bar{s}} \rightarrow \Spec k(\bar{s})[x,y,z_1,\cdots, z_{r-1}]/(xy)
\end{align*} 

\noindent These families of spaces are required to admit log structures $N^{W/S}$ on $S$ and $M^{W/S}$ on $W$, such that $M^{W/S}$ is given by the cocartesian diagram

\begin{align*}
\xymatrix{\ar @{} [dr] |{}
\NN \ar[d] \ar[r] & \NN^2 \ar[d] \\
N^{W/S} \ar[r] & M^{W/S} }
\end{align*}

\noindent and such that the morphism $W \rightarrow S$ is in fact a log morphism $(W,M^{W/S}) \rightarrow (S,N^{W/S})$. We will further require that $N^{W/S}$ is locally free. Its rank at $s \in S$ equals the number of irreducible components of the singular locus of the fiber $W_s$. We further require that the spaces $W$ come equipped with a map $W \rightarrow X$.

\end{definition}

\begin{definition}
A \emph{twisted expanded} log FM type space of $X$ is a log morphism $(W,M) \rightarrow (S,N)$, where

\begin{itemize}
\item $W \rightarrow S$ is as in Definition 3 above, and all relevant logarithmic data are obtained from a cartesian diagram

\begin{align*}
\xymatrix{\ar @{} [dr] |{}
(W,M) \ar[d] \ar[r] & (W,M^{W/S}) \ar[d] \\
(S,N) \ar[r] & (S,N^{W/S}) }
\end{align*}

\noindent that is, the logarithmic data simply corresponds to a morphism of log structures $t^{W/S}:N^{W/S} \rightarrow N$. We will reserve the notation $t^{W/S}$ to always indicate this morphism and denote it by $t$ to simplify notation when no confusion may arise. \\
\item There is a chart for the morphism $t^{W/S}:N^{W/S} \rightarrow N$ of the form

\begin{align*}
\xymatrix{\ar @{} [dr] |{}
\NN^{m} \ar[d] \ar[r] & \NN^{m}\oplus \NN^{m'} \ar[d] \\
N^{W/S} \ar[r] & N  }
\end{align*}

\noindent Here the top map is of the form $(d,0)$, where $d=(d_1,\cdots,d_m)$ is a diagonal matrix of natural numbers. \\

\end{itemize}

\end{definition}

We will refer to twisted expanded log FM type spaces of $X$ simply as log FM spaces for brevity, contrary to the literature. Having defined both the sources and the targets of log stable maps, we can now give the definition of the central objects of study of this paper. \\

\begin{definition} \label{def5}A log stable map from a log curve $(C,M)/(S,N)$ to an FM space $(W,K)/(S,N)$ is a log morphism $f:(C,M) \rightarrow (W,K)$ over $(S,N)$ such that, over each point $s \in S$

\begin{itemize}
\item The cokernel of the map $N^{W/S}_s \rightarrow N_s$ has rank equal to the number of non-distinguished nodes. \\
\item The map $f^{*}K_s \rightarrow M_s$ is simple at the distinguished nodes. \\
\item Stability: The automorphism group $Aut(f_s)$ is finite.\\
\item The following minimality condition holds: Either $(C,M)/(S,N)$ is a minimal log curve; if not, then $N = \NN \oplus N'$, with $N^{C/S} \rightarrow N'$ minimal and $N^{W/S} \rightarrow \NN$ surjective. 
\end{itemize}

\end{definition}

We explain the terminology: Over each $s\in S$, $C_s$ is a nodal curve and $W_s$ an FM type space. A node of $C_s$ is called distinguished if it maps into the singular locus of $W_s$ and non-distinguished otherwise. A morphism between locally free log structures is called simple if it is given by a diagonal matrix, as in the definition of FM spaces above. An automorphism of $f:(C,M) \rightarrow (W,K)$ over $(S,N)$ is a Cartesian diagram over $(S,N)$ \\

\begin{align*}
\xymatrix{\ar @{} [dr] |{}
(C,M) \ar[d] \ar[r] & (W,K) \ar[d] \\
(C,M) \ar[r] & (W,K)  }
\end{align*}

\noindent that respects the map to $X$, that is, on the level of underlying schemes we have

\begin{align*}
\xymatrix{\ar @{} [dr] |{}
C \ar[d] \ar[r] & W \ar[d] \ar[r] & X \ar[d]^{=} \\
C \ar[r] & W \ar[r] & X}  
\end{align*}

\begin{remark}\label{rmk1} The minimality condition slightly deviates from Kim's definition. In \cite{Kim}, the definition of a log stable map requires that the log curve $(C,M) \rightarrow (S,N)$ is minimal. This neglects the possibility that there is no node in the curve $C$ mapping to the original divisor $D=D[0]$. The typical example of this situation is in $(X,D)=(\PP^2,\PP^1)$, where a line $\PP^1$ rotates to collapse into $D$: the limit map is the map from $\PP^1$ to $X[1]$ that sends all of $\PP^1$ into $X[1]-X$, with the image line intersecting $D[1]$ transversely. There is a more satisfying, intrinsic explanation of the minimality condition. Minimality is a categorical property: Minimal log schemes $(S,N)$ are precisely the log schemes one must restrict to in order to be able to consider a stack over log schemes, such as $\mathcal{M}_{g,n}^{\textup{log}}$ or the stack of all log maps into FM spaces, as a log stack. This is the content of the paper \cite{GMinimality}. Kim's minimal log curves are precisely the minimal objects for the stack $\mathcal{M}_{g,n}^{\textup{log}}$. However, the minimal objects for the stack of all log maps into FM spaces that satisfy the first three properties include the case that no node of the curve maps into $D[0]$, and we must include this case in the definition. Here we also remark that what we call minimal here is the same thing as what is called basic in \cite{AC},\cite{GS}. Both notions correspond to the same categorical notion of minimality.
\end{remark}

Let us describe the logarithmic data concretely in the case when $S = \Spec k$ is a geometric point. This description will be useful in the sequel. The locally free log structures on $S$ is free in this case, described by a chart $\NN^r$ for certain integers $r$; specifically, we have

\begin{align*}
N^{C/S} &= \NN^{m''} \oplus \NN^{m'} \oplus k^{\ast}\\
N^{W/S} &= \NN^{m} \oplus k^{*}\\
N &= \NN^{m} \oplus \NN^{m'} \oplus k^{*}
\end{align*}

\noindent The morphisms $s^{C/S}:=s,t^{W/S}:=t$ are described on the level of characteristic monoids as follows: $\bar{t}:\bar{N}^{W/S} \rightarrow \bar{N}$ is given by a diagonal matrix of the form $(d_1,\cdots,d_m,0,\cdots,0)$, as above. The morphism $\bar{s}:\bar{N}^{C/S} \rightarrow \bar{N}$ is given by a matrix of the form $(\Gamma, id)$, where $\Gamma$ is a generalized diagonal matrix:

\begin{align}
\begin{pmatrix}
\Gamma_{1,1} \cdots \Gamma_{1,k_1} & 0 \cdots 0 & \cdots  & 0 \cdots 0 \\
0 \cdots 0 & \Gamma_{2,1} \cdots \Gamma_{2,k_2} & \cdots  & 0 \cdots 0 \\
\vdots & \vdots & \ddots & \vdots \\
0 \cdots 0 & 0 \cdots 0 & \cdots & \Gamma_{m,1} \cdots \Gamma_{m,k_m}
\label{(1)}
\end{pmatrix}
\end{align}

\noindent The integers $k_1, \cdots, k_m$ add up to $m''$. It is possible that the first row of the matrix is $0$, in which case $d_1=1$. This happens when no node of the curve maps to the original divisor $D=D[0]$. Otherwise, the log curve $(C,M) \rightarrow (S,N)$ is minimal, which means that there is no common divisor between the integers $\Gamma_{i,k_1}, \cdots, \Gamma_{i,k_i}$. For each $\Gamma_{i,j}$, there is an integer such that $d_i = \Gamma_{i,j}l_{i,j}$. In other words, there is a commutative diagram

\begin{align*}
\xymatrix{\ar @{} [dr] |{}
N^{W/S} \ar[d] \ar[r] & N^{C/S} \ar[d] \\
N \ar[r] & N  }
\end{align*}

We will now fix a stack $\mathcal{B}$ of certain log FM type spaces, and denote by $\mathcal{B}^{\textup{etw}}$ the stack whose objects are FM spaces (i.e twisted expanded log FM type spaces) whose underlying spaces are in $\mathcal{B}$. We denote by $\mathcal{U}$ and $\mathcal{U}^{\textup{etw}}$ the universal family of $\mathcal{B}$ and $\mathcal{B}^{\textup{etw}}$ respectively. In other words, we consider spaces $W \in \mathcal{B}$ but endow them with log structures as above. We will consider the stack $\Kim$ of log stable maps to targets in $\mathcal{B}^{\textup{etw}}$. It is proven \cite{Kim} that if the stack $\mathcal{B}$ is algebraic, $\Kim$ is also algebraic.  \\

\begin{remark} Let us at this point explain the connection with Jun Li's original definitions and clarify this concept geometrically. A family of expansions $W \rightarrow S$ of a pair $(X,D)$, or similarily, of a $D$-semistable degeneration $X=Y_1 \cup_D Y_2$ has canonical log structures that determine FM spaces. The canonical log structure on a family of expansions is obtained in a manner formally identical to the way that the canonical log structure on a nodal curve is obtained. A detailed treatment of the canonical log structures on expansions can be found in Olsson's paper \cite{OlssonSemistable}. Briefly, there is an algebraic stack $\mathcal{B}$ parametrizing expansions. The stack $\mathcal{B}$ is in fact the open substack of the stack $\mathfrak{M}_{0,3}$ of $3$ marked, genus $0$ prestable curves where the first two markings are on the first component of the curve and the third marking on the last - see for instance \cite{GV}. In $\mathcal{B}$ there is a normal crossings divisor corresponding to singular expansions; therefore, $\mathcal{B}$ admits a log structure $\mathcal{M}_{\mathcal{B}}$. Similarily, the universal family $\mathcal{U}$ over $\mathcal{B}$ admits a log structure $\mathcal{M}_{\mathcal{U}}$. A family of expansions corresponds to a cartesian diagram

\begin{align*}
\xymatrix{\ar @{} [dr] |{}
W \ar[d] \ar[r]^{u} & \mathcal{U} \ar[d] \\
S \ar[r]_{t} & \mathcal{B} }
\end{align*}

\noindent The pullback log structures $t^{*}\mathcal{M}_{\mathcal{B}}$ and $u^{*}\mathcal{M}_{\mathcal{U}}$ on $S$ and $W$ are what we denoted by $N^{W/S}$ and $M^{W/S}$ above. Therefore, expansions are examples of FM spaces. We may thus consider log stable maps to expansions. The underlying morphism of schemes of such a log stable map is a relative stable map in the sense of Jun Li; the predeformability condition is enforced by the requirement that the map is a map of log schemes. The log structures are thus additional algebraic data on a relative stable map. The log structures encode essential geometric information very conveniently. Suppose for simplicity that $S=\Spec{k}$ is a geometric point. We have seen above the form of the log structures $N^{C/S},N^{W/S},N$ and the maps between them. The rank of $N^{W/S}$, which we denoted by the number $m$ above, indicates that the target is the $m$-th expansion $(X[m],D[m])$ of $(X,D)$. The number $m'$ is the number of nondistinguished nodes. The number $m''$ is the number of distinguished nodes. The matrix $\Gamma$ above indicates that $k_1$ of the distinguished nodes map to the first singular locus in $(X[m],D[m])$ (namely to $D[0])$, $k_2$ map to $D[1]$, and so forth. The contact order of the $j$-th node mapping to the $i$-th singular locus is $l_{i,j}$. Note that once the underlying stable map is fixed, the diagram

\begin{align*}
\xymatrix{\ar @{} [dr] |{}
\bar{N}^{W/S} \ar[d]_{\bar{t}^{W/S}} \ar[r] & \bar{N}^{C/S} \ar[d]^{\bar{s}^{C/S}} \\
\bar{N} \ar[r] & \bar{N}  }
\end{align*}

\noindent between the characteristic monoids of the log structures is determined. This means that in order to determine the full diagram

\begin{align*}
\xymatrix{\ar @{} [dr] |{}
N^{W/S} = \bar{N}^{W/S} \oplus k^{*} \ar[d]_{t^{W/S}} \ar[r] & N^{C/S} = \bar{N}^{C/S} \oplus k^{*} \ar[d]^{s^{C/S}} \\
N = \bar{N} \oplus k^{*} \ar[r] & N  }
\end{align*}

\noindent we need to determine the elements of $k^{*}$ to which the generators of $N^{C/S}$ and $N^{W/S}$ are mapping to. In fact, all generators of $N^{W/S}$ may be chosen to map into $1 \in k^{*}$ after automorphism, so it is enough to treat only $N^{C/S}$; matrix 1 in remark 1 indicates that $s^{C/S}$ has the following form: 

\begin{align*}
e_{ij} \mapsto (\Gamma_{ij} e_i,u_{ij})
\tag{1'} \label{eq:1'}
\end{align*}

\noindent Note however that the units $u_{ij}$ are restricted: they must satisfy the equation $u^{l_{ij}}=1$, in order for the diagram 

\begin{align*}
\xymatrix{\ar @{} [dr] |{}
(C,M) \ar[d] \ar[r]^f & (W,M^W) \ar[ld]\\
(S,N) &  }
\end{align*}

\noindent to commute. This shows that there is a finite number of ways to give to a relative stable map the structure of a log stable map. In other words, if $\Li$ denotes Jun Li's space of expansions and $\Kim$ Kim's space of log stable maps to expansions of $(X,D)$, which is algebraic stack since the stack of expansions $\mathcal{B}$ is algebraic, there is a forgetful morphism

\begin{align*}
\Kim \rightarrow \Li
\end{align*}

For the rigorous definition of the morphism $\pi^\flat:\Kim \rightarrow \Li$ we refer to the work of Gross and Siebert \cite{GS} and the paper \cite{AMW} of Abramovich, Marcus and Wise. This is the left vertical arrow of diagram \ref{eq:1} of the introduction. The fact that relative stable maps are a locally closed substack of the stack of all maps expresses the fact that the predeformability condition is locally closed. The fact that the stack of log stable maps is open in the stack of all log maps expresses the fact that predeformability is enforced by requiring the map from a nodal curve to an expansion be a log map. 
\end{remark}

\end{section}

\begin{section}{Equivariant Embedding}
Since Kim's moduli space does not carry a fundamental class, but rather a virtual fundamental class, in order to prove a localization formula we need to use the virtual localization formula of Graber-Pandharipande \cite{GP}. To use their results, we need to establish the technical condition:

\begin{thm}
There is a locally closed equivariant immersion of $\Kim$ into a smooth Deligne-Mumford stack.
\label{thm:1}
\end{thm}

We will do this by proving that $\Kim$ satisfies a slightly stronger condition, also shared by Jun Li's space $\Li$. Namely we will prove that $\Kim$ satisfies the following property, which we will abbreviate as property SE (for strong embedding property): \\

\begin{itemize}
\item $\Kim = [V/G]$ is a global quotient, where $G$ is a reductive group, $V$ is a locally closed subset of a smooth projective $W$ with an action of $G$ extending that of $V$. \\
\item There is a $\CC^{*} \times G$ action on $W$ which descends to the $\CC^{*}$ action on $[V/G]$ \\
\end{itemize}

In \cite{GV}, it is shown that $\Li$ satisfies SE by an explicit construction. We have seen there is a morphism $\Kim \rightarrow \Li$; this morphism is in fact finite, as shown in lemma 2 below. Therefore, it suffices to show the following lemma; the idea of the proof is due to Vistoli.\\

\begin{lem} Suppose $f:X \rightarrow Y$ is a $\CC^{*}$ equivariant finite morphism between Deligne-Mumford stacks, and assume that $Y$ satisfies SE. Then $X$ satisfies SE as well, and thus embeds $\CC^{*}$ equivariantly into a smooth Deligne-Mumford stack.
\end{lem}

\begin{proof} Since $Y$ satisfies SE, we may write $Y=[V/G] \subset [W/G]$ with $V \subset W$ a $\CC^* \times G$ equivariant locally closed subset and $W$ a smooth projective variety. We can find a $\CC^* \times G$ equivariant open smooth subvariety $W_{o} \subset W$ such that:
\begin{enumerate}
\item$G$ acts on $W_0$ with reduced finite stabilizers
\item $V$ is a closed subvariety of $W_0$.
\end{enumerate}

(Note that this is clearly possible since $[V/G]$ is Deligne Mumford by the hypothesis, or equivalently $G$ acts on $V$ with finite reduced stabilizers.) The composed morphism $f:X \rightarrow Y \rightarrow [W_{o}/G]$ is then still $\CC^*$ equivariant and finite, so by replacing $V$ with $W_{o}$ we may assume $V$ is smooth. Since the morphism $X \rightarrow Y$ is $\CC^{*}$ equivariant, we obtain a morphism $[X/\CC^{*}] \rightarrow [Y/\CC^{*}]$. Here $[X/\CC^{*}]$ denotes the stack quotient in the sense of Romagny \cite{R}, though our notation is slightly different. Assume for a moment that $[Y/\CC^{*}]$ has the resolution property.  Then, the sheaf $\bar{f}_{*}(\mathcal{O}_{[X/\CC^{*}]})$, which is coherent since $f$ is finite, is the quotient of a locally free sheaf $\mathcal{E}$ on $[Y/\CC^{*}]$. In other words, $[X/\CC^{*}]$ embeds into a vector bundle over $[Y/\CC^{*}]$. That is, there is a $\CC^{*}$ equivariant morphism of $X$ into a $\CC^{*}$ equivariant vector bundle over $Y=[V/G]$, hence a $\CC^{*}$ equivariant  embedding of $X$ into a smooth DM stack. In particular, $X$ admits a $\CC^{*}$ equivariant embedding into a stack of the form $[U/G]$, where $U$ is a smooth vector bundle over $V$, so it is, in fact, a quotient. To see that $X$ satisfies SE, then, the only thing that remains to be shown is that $U$ embeds $\CC^* \times G$ equivariantly into a smooth and projective variety as a locally closed subset. We give the proof after proving the resolution property for $[Y/\CC^*]$, as it requires essentially the same argument.\\

We thus show that $[Y/\CC^{*}]$=$[[V/G]/\CC^{*}]$ does, indeed, have the resolution property. Equivalently, we prove that any $\CC^{*}$ equivariant coherent sheaf $\mathcal{F}$ on $[V/G]$  admits a $\CC^{*}$ equivariant surjection from a $\CC^{*}$ equivariant locally free sheaf.\\ 

Let $p:V\rightarrow [V/G]$ be the projection. Since we assumed above that $V$ is smooth and quasiprojective, by  \cite[Corollary 1.6]{GIT}, we can find a $G\times\CC^*$ immersion $i:V\rightarrow \PP^n$. Let $\mathcal {O}_V(1)=i^*\mathcal{O}_{\PP^n}(1)$ and consider a $\CC^*$ equivariant coherent sheaf $\mathcal {F}$ on $[V/G]$. Since the $G\times \CC^*$ action on $V$ descends to the $\CC^*$ action on the quotient $[V/G]$, the pullback $p^*\mathcal{F}$ is $G\times \CC^*$ equivariant. It is also coherent. Therefore, since $V$ is quasiprojective, we may pick a large $N$ so that the twisted sheaf $p^*\mathcal{F}(N)$ is generated by global sections, say $s_1, s_2, \cdots s_m \in H^{0}(V,p^*\mathcal{F}(N))$. Let $V_1=<{s_1,s_2,\dots, s_m}>$ be their linear span in $H^0(V,p^*\mathcal{F}(N))$. Following the argument in the proof of \cite[ pp.25-26, Lemma $^*$]{GIT}, we can find \\
\begin{center}
$V_1\subset V_2\subset H^{0}(V,p^*\mathcal{F}(N))$
\end{center}
\noindent where $V_2$ is a $\CC^{*} \times G$ equivariant finitely generated subspace. We therefore obtain a natural $\CC^* \times G$ equivariant surjection\\
~\\
\begin{center}
$V_2 \rightarrow p^*\mathcal{F}(N) $
\end{center}
~\\
\noindent hence a natural $\CC^* \times G$ equivariant surjection
~\\
\begin{center}
$V_2 (-N) \rightarrow p^*\mathcal{F}$
\end{center}
~\\
\noindent which  descends to a $\CC^*$ equivariant surjection from a $\CC^*$ equivariant locally free sheaf on $[V/G]$ to $\mathcal{F}$ . This shows that $[Y/\CC^*]$ has the resolution property. Note that this argument also suffices to complete the proof that $X$ satisfies SE: we have already seen that $X$ embeds into a vector bundle $[U/G]$; the argument just given shows that $U$ admits a $\CC^* \times G$ equivariant surjection $(\mathcal{O}_V(-M))^k \rightarrow U$ for some integers $M$,$k$. Therefore, $[U/G]$ admits a $\CC^*$ equivariant locally closed immersion to the projectivized bundle $\PP(\mathcal{O}_{\PP^n}(-M)^k \oplus \mathcal{O}_{\PP^n})$ over ${\PP}^n$, that is, a smooth projecive variety.  

\end{proof}

\begin{lem}\label{lem2}
The morphism $\Kim \rightarrow \Li$ is finite. 
\end{lem}

\begin{proof} We need to show that every geometric point of $\Li$, that is, every relative stable map $\und{f}$, has a finite number of preimages, and that the morphism is representable: for each preimage $f$, the map $Aut(f) \rightarrow Aut(\und{f})$ is injective. The discussion in section 2, or Remark 6.3.1 in \cite{Kim} establishes the finiteness of the preimages. Let then $f$ be a log stable map lying over $\und{f}$. Denote the base by $(S,N)$, where $S=\Spec k$, the source curve by $C$, and the target by $W=X[m'']$. Let $m,m'$ be the number of distinguished and non-distinguished nodes of $f$ respectively, and $m''$ the number of nodes of $W$. We have $N^{C/S} = \NN^m \oplus \NN^{m'} \oplus k^{*}$, $N^{W/S} = \NN^{m''} \oplus k^{*}$, $N = \NN^{m''} \oplus \NN^{m'} \oplus k^{*}$. An automorphism of $f$ lying over the identity automorphism of $\und{f}$ is simply an automorphism of the logarithmic structures, that is, an automorphism $A$ of $N$ that respects the maps $s:=s^{C/S}:N^{C/S} \rightarrow N$, $t:=t^{W/S}:N^{W/S} \rightarrow N$. In other words, we are looking for commutative diagrams

\begin{align*}
\xymatrix{\ar @{} [ldr] |{}
\NN^m \oplus \NN^{m'} \oplus k^{*} \ar[r]^s \ar[d]_{=} & \NN^{m''} \oplus \NN^{m'} \oplus k^{*} \ar[d]^{A} \\
\NN^m \oplus \NN^{m'} \oplus k^{*} \ar[r]_s& \NN^{m''} \oplus \NN^{m'} \oplus k^{*}}, 
\xymatrix{\ar @{} [ldr] |{}
\NN^{m''} \oplus k^{*} \ar[r]^-t \ar[d]_{=} & \NN^{m''} \oplus \NN^{m'} \oplus k^{*} \ar[d]^{A} \\
\NN^{m''} \oplus k^{*} \ar[r]_-t& \NN^{m''} \oplus \NN^{m'} \oplus k^{*}}, 
\end{align*}

\noindent The induced matrix $\bar{s}$ on the level of characteristic monoids is a generalized diagonal matrix which is the identity on the last $m'$ components; and similarily $\bar{t}$ is a diagonal matrix with finite cokernel on the first $m$ components. Therefore $\bar{A}$ must be a diagonal matrix as well - in fact, the identity on characteristic monoids. It follows that each factor of $\NN^{m''}$ contributes to automorphisms separately, and we may thus assume $m''=1$ without loss of generality. In other words, the automorphism group of $N$ splits as a product of the automorphism groups contributed by each node of the target. We are thus reduced to studying three cases: (a) either there are $m$ distinguished nodes mapping to the node of $W$; (b) there is a node in $W$ but no node in $C$, i.e $m=m'=0$ (cf the minimality condition of definition $5$); (c) or, there is one non-distinguished node and no distinguished node, i.e $m=0,m'=1$, since every non-distinguished node contributes precisely one factor of $\NN$ in $N$. The first case is most interesting. In this case, the two diagrams take the form

\begin{align*}
\xymatrix{\ar @{} [ldr] |{}
\NN^m  \oplus k^{*} \ar[r]^s \ar[d]_{=} & \NN \oplus k^{*} \ar[d]^{A} \\
\NN^m  \oplus k^{*} \ar[r]_s& \NN \oplus k^{*}}, 
\xymatrix{\ar @{} [ldr] |{}
\NN \oplus k^{*} \ar[r]^-t \ar[d]_{=} & \NN \oplus k^{*} \ar[d]^{A} \\
\NN \oplus k^{*} \ar[r]_-t& \NN \oplus k^{*}}, 
\end{align*}

\noindent These are homomorphisms of log structures, thus lie over the map to the field $k$, and the factor of $k^*$ maps identically to itself. The automorphism $A$ is determined by its action on the generator $e$ of $\NN$ and thus takes the form $e \mapsto (e,v)$ for a unit $v \in k^*$. On the other hand, we have seen in formula \ref{eq:1'} that 

\begin{align*}
e_i \mapsto (\Gamma_ie,u_i)
\end{align*}

\noindent and thus commutativity of the diagram implies $(\Gamma_ie,u_iv^{\Gamma_i}) = (\Gamma_ie,u_i)$, that is, $v^{\Gamma_i}=1$ for all $i$. By minimality, the greatest common divisor of the $\Gamma_i$ is $1$, and thus $v=1$. \\

In the case (b), the diagrams become  

\begin{align*}
\xymatrix{\ar @{} [ldr] |{}
k^{*} \ar[r]^-{id} \ar[d]_{=} & \NN \oplus k^{*} \ar[d]^{A} \\
k^{*} \ar[r]_-{id}& \NN \oplus k^{*}}, 
\xymatrix{\ar @{} [ldr] |{}
\NN \oplus k^{*} \ar[r]^t \ar[d]_{=} & \NN \oplus k^{*} \ar[d]^{A} \\
\NN \oplus k^{*} \ar[r]_t& \NN \oplus k^{*}}, 
\end{align*}

\noindent and the map $t$ is of the form $e \mapsto (de,1)$. However, minimality requires that $d=1$. Hence, $A$, which has the form $e \mapsto (e,v)$ must have $v=1$ and is thus trivial. Finally, in case (c) the diagrams become 

\begin{align*}
\xymatrix{\ar @{} [ldr] |{}
\NN \oplus k^{*} \ar[r]^-{s} \ar[d]_{=} & \NN \oplus k^{*} \ar[d]^{A} \\
\NN \oplus k^{*} \ar[r]_-{s}& \NN \oplus k^{*}}, 
\xymatrix{\ar @{} [ldr] |{}
k^{*} \ar[r]^-{id} \ar[d]_{=} & \NN \oplus k^{*} \ar[d]^{A} \\
k^{*} \ar[r]_-{id}& \NN \oplus k^{*}}, 
\end{align*}

\noindent and $s$ is the identity on the level of characteristics, so $A$ is trivial as well, by the same argument as in (b).

\end{proof}
 
From the two lemmas it follows that the localization formula of \cite{GP} can be applied to $\Kim$. \\

\end{section}

\begin{section}{The Obstruction Theory}
In this section we analyze the obstruction theory of $\Kim$. We will write $M$ for $\Kim$ to ease the notation. There is a forgetful morphism $\tau: M \rightarrow \mathcal{MB}$,  where

\begin{align*}
\mathcal{MB} = \mathfrak{M}_{g,n}^{\textup{log}} \times_{\textup{LOG}} \mathcal{B}^{\textup{etw}}
\end{align*}

Here, the category $\textup{LOG}$ is the category whose objects are log schemes and morphisms are strict log morphisms. The stack $\mathcal{MB}$ parametrizes, over a scheme $S$, pairs $(C,M)/(S,N)$ and $(W,K)/(S,N)$ of a $n$-marked log curve over $(S,N)$ and a FM space in $\mathcal{B}^{\textup{etw}}$ over the same log scheme $(S,N)$. The morphism $\tau$ sends a log stable map to the pair consisting of the source of the map and the target; it is the forgetful morphism forgetting the data of the map. The stack $\mathcal{MB}$ is the analogue of the Artin stack $\mathfrak{M}$ of prestable curves in ordinary Gromov-Witten theory. It is also smooth and, in fact, log smooth, as it further has the structure of a log stack. \\

By standard properties of the cotangent complex, the morphism $\tau$ induces a distinguished triangle

\begin{align*}
\tau^{-1} \mathcal{L}_{\mathcal{MB}} \rightarrow \mathcal{L}_M \rightarrow \mathcal{L}_{M/\mathcal{MB}} \rightarrow
\end{align*}

Consider the diagram

\begin{align*}
\xymatrix{\ar @{} [dr] |{}
\mathcal{C}^{\textrm{univ}} \ar[d]_{\pi} \ar[r]^{f} & \mathcal{U}^{\textrm{etw}} \ar[d] \\
M \ar[r] & \mathcal{B}^{\textrm{etw}} }
\end{align*}

\noindent Here $\mathcal{C}^{\textrm{univ}}$ is the universal family of $M$ and $\mathcal{U}^{\textrm{etw}}$ the universal family over $\mathcal{B}^{\textrm{etw}}$. The morphism $f$ is the evaluation map and $\pi$ the projection. It is proven in section 7 of \cite{Kim} that there is a canonical morphism

\begin{align*}
E^{\bullet} = (R\pi_{*}f^{*}T_{\mathcal{U}^{\textrm{etw}}/\mathcal{B}^{\textrm{etw}}}^{\textup{log}})^{\vee} \rightarrow \mathcal{L}^{\log}_{M/\mathcal{MB}}
\end{align*}

\noindent which is a (relative) perfect obstruction theory. The morphism $\tau$ is strict, therefore the log cotangent complex $\mathcal{L}^{\log}_{M/\mathcal{MB}}$ coincides with the ordinary cotangent complex $\mathcal{L}_{M/\mathcal{MB}}$. Furthermore, the stack $\mathcal{MB}$ is smooth, therefore $\mathcal{L}_{\mathcal{MB}}$ is a two-term complex concentrated in degrees $0$ and $1$ (it is an Artin stack, so it has automorphisms). We therefore have a diagram

\begin{align*}
\xymatrix{\ar @{} [dr] |{}
\mathcal{L}_{M/\mathcal{MB}}[-1] \ar[r] & \tau^{-1} \mathcal{L}_{\mathcal{MB}} \ar[r] & \mathcal{L}_M \ar[r] & \mathcal{L}_{M/\mathcal{MB}}  \\
E^{\bullet}[-1] \ar[u] \ar[r] & \tau^{-1} \mathcal{L}_{\mathcal{MB}} \ar[u] &  & E^{\bullet} \ar[u] }
\end{align*}

We may fill in the lower row by the cone of $E^{\bullet}[-1] \rightarrow \tau^{-1} \mathcal{L}_{\mathcal{MB}}$ to obtain an (absolute) perfect obstruction theory for $M$. Applying the functor $RHom(\bullet,\mathcal{O}_{M})$, we obtain the following lemma:\\

\begin{lem} Over a geometric point $f:(C,M_C)/(k,N) \rightarrow (W,M_W)/(k,N)$, the tangent space $\mathcal{T}^1$ and obstruction space $\mathcal{T}^2$ of $M$ fit into an exact sequence
\begin{align*}
0 \rightarrow \mathfrak{aut}(\tau(f)) \rightarrow H^{0} (f^{*}T^{\log}_W) \rightarrow \mathcal{T}^1 \rightarrow \textrm{Def}(\tau(f)) \rightarrow H^{1}(f^{*}T^{\log}_W) \rightarrow \mathcal{T}^2 \rightarrow 0.
\end{align*}
\end{lem}

The term $\mathfrak{aut}(\tau(f))$ refers to the group of first order infinitesimal automorphisms of $\tau(f)$. To carry out localization calculations, we need to know the equivariant Euler classes of the sheaves $\mathcal{T}^1,\mathcal{T}^2$. Since the Euler class is a K-theory invariant, it is enough to understand the other four terms in the exact sequence. The terms $H^{k} (f^{*}T^{\log}_W)$, $k=0,1$ are the cohomology groups of explicit locally free sheaves on the curve $C$, which may be calculated by hand. In fact, more can be said: 

\begin{lem}
Suppose $\pi: W \rightarrow X$ denotes the canonical contraction map. Then $T^{\log}_W = \pi^{*}T^{\log}_X = \pi^{*}T_X(-D)$.  
\end{lem}

\begin{proof} 
We recall the explicit construction of the log schemes $X[n]$. For details of the construction, the reader is referred to \cite{Li}, \cite{OlssonSemistable}. Set $Y[0]=X[0]=X$; let $Y[1]$ be the blow up of $Y[0] \times \Aff^1 = X \times \Aff^1$ along the divisor $D \times 0$, with a divisor $D[1] \subset Y[1]$ defined as the proper transform of $D \times \Aff^1$. We view $Y[1]$ as a family of log schemes over $\Aff^1$, with logarithmic structure on $\Aff^1$ coming from the divisor $0 \in \Aff^1$, the product log structure on $X \times \Aff^1$, and the natural log structure on the blow up $Y[1]$. The log scheme $X[1]$ over $\Spec k$ is the fiber of this family over $0$, with the induced logarithmic structures on $X[1]$ and the base $\Spec k$. Next, $Y[2]$ is the blow up of $Y[1] \times \Aff^1$   along $D[1] \times 0$, with a divisor $D[2]$ the proper tranform of $D[1] \times \Aff^1$. This is viewed as a family over $\Aff^2$ , with the standard toric log structure on $\Aff^2$ coming from the axes, and the log structure on $Y[2]$ again arising from the product log structure on $Y[1] \times \Aff^1$ and blowing up. The scheme $X[2]$ is the fiber over $0 \in \Aff^2$. In general, $Y[n]$ is constructed from $Y[n-1] \times \Aff^1$ by blowing up $D[n-1] \times 0$; this is viewed as a family over $\Aff^n$ with the toric logarithmic structure on $\Aff^n$ and the logarithmic structure on $Y[n]$ coming from $Y[n-1] \times \AA^1$ and blowing up; $D[n]$ is the proper transform of $D[n-1] \times \Aff^1$; and $X[n]$ is the fiber over $0 \in \Aff^n$. Since blowing up along a divisor yields a log \'{e}tale map to the originial scheme, the morphism $\pi$ in the diagram

\begin{align*}
\xymatrix{\ar @{} [dr] |{}
Y[n] \ar[r]^-{\pi} \ar[d] & X \times \Aff^n \ar[d]  \\
 \Aff^n  \ar[r]_{=} & \Aff^n }
\end{align*}

\noindent is log \'{e}tale. It follows that the log tangent bundle of $Y[n]$ over $\AA^n$ is the pullback of the log tangent bundle of $X \times \Aff^n$ over $\Aff^n$, which is simply the log tangent bundle of $X$. Since log \'{e}tale morphisms are stable under base change, it follows that the tangent bundle of $X[n]$ over the base $\Spec k$ is pulled back from $X$ as well, as claimed. 
\end{proof}
 
Thus the two terms $H^{k} (f^{*}T^{\log}_W)$ only depend on the logarithmic map from $C$ to $X$. What we have to understand are the two terms $\mathfrak{aut}(\tau(f))$, $\textrm{Def}(\tau(f))$, that is, the infinitesimal automorphism group and tangent space of a point of $\mathcal{MB}$. In the discussion that follows, \emph{we restrict attention to the stack $\mathcal{MB}$ when $\mathcal{B}$ is the stack of expansions} of $X$, discussed in remark 2.   \\

We will understand the deformation group in terms of the stack of twisted stable curves $\mathfrak{M}_{g,n}^{\textup{tw}}$ of Abramovich-Vistoli \cite{AV}, which is well understood. To do so, we must digress a bit. First, it will be easier for technical reasons to compare the deformation theory of $\mathcal{MB}$ with the deformation theory of the stack of log twisted curves of Olsson \cite{Otw}, which we denote by $\mathfrak{M}_{g,n}^{\textup{logtw}}$. It is shown in \cite{Otw} that $\mathfrak{M}_{g,n}^{\textup{logtw}} \cong \mathfrak{M}_{g,n}^{\textup{tw}}$. Recall the definition of $\mathfrak{M}_{g,n}^{\textup{logtw}}$: \\

\begin{definition} A log twisted curve over a scheme $S$ is a log curve $C \rightarrow S, N^{C/S} \stackrel{\alpha}{\rightarrow} N$ where $N$ is locally free and $\alpha$ is an injection that is locally given by a diagonal matrix. 
\end{definition}  

Let us fix some notation. Denote the set of non-distinguished nodes of the curve by $R$, the set of distinguished nodes by $S$, and the nodes of the target by $T$. In the notations above, we would have $|T|=m,|R|=m',|S|=m''$. Furthermore, let

\begin{align*}
\mathcal{A} = [\Aff^1/\mathbb{G}_m]
\end{align*}

\noindent denote the ''universal target'': this is the moduli space that over a scheme $S$ parametrizes line bundles $L$ over $S$, together with a section $s \in \Gamma (S,L)$, up to isomorphism. \\

Consider now a family $F$ of log stable maps over a scheme $S$, which specializes over a geometric point $\Spec k$ to a map $f:((C,M_C)/(k,N),\vec{x}) \rightarrow (W,M_W)/(k,N)$. We must now define several morphisms. \\

First, consider the image of $F$ in $\mathfrak{M}_{g,n}^{\textup{log}}$ under the natural forgetful morphism, which forgets the data of the target and the map. \'{E}tale locally around $(C,M_C)/(k,N)$, we have a map $\mathfrak{M}_{g,n}^{\textup{log}} \rightarrow \mathfrak{M}_{g,n}^{\textup{logtw}}$. The morphism is defined as follows; for an \'{e}tale neighborhood $(C',M_{C'})/(S',N)$ of $(C,M_C)/(k,N)$ in which all log structures $N^{C'/S'}$ and $N$ are actually free, the map $s^{C'/S'}: N^{C'/S'} \rightarrow N$ factors as $N^{C'/S'} \stackrel{r^{C'/S'}}{\rightarrow} N^{C'/S'} \rightarrow N$. This factorization has the following description on the level of characteristic monoids: we have seen in formula \ref{(1)} of section 2 that the morphism $\bar{N}^{C'/S'}=\NN^{m''} \oplus \NN^{m'} \rightarrow \bar{N}=\NN^{m} \oplus \NN^{m'}$ has the form $(\Gamma, id)$, with $\Gamma$ a generalized diagonal matrix. This factors as  

\begin{align*}
\NN^{m''} \oplus \NN^{m'} \stackrel{(\gamma,id)}{\rightarrow} \NN^{m''} \oplus \NN^{m'} \stackrel{p}{\rightarrow} \NN^{m} \oplus \NN^{m'}
\end{align*}

where $\gamma$ is the matrix $\Gamma$ 'made diagonal', i.e 

\begin{align*}
\gamma =
\begin{pmatrix}
\Gamma_{1,1} & 0 & 0 & \cdots & 0 & 0 & \cdots & 0 \\
0 & \Gamma_{1,2} & 0 & \cdots & 0 & 0 & \cdots & 0 \\
\vdots & \vdots & \ddots & \vdots & \vdots &\vdots &\vdots &\vdots \\
0 & 0 & \cdots & \Gamma_{1,k_1} & 0 & \cdots & \cdots & 0 \\
0 & 0 & \cdots & 0 & \Gamma_{2,1} & 0 & \cdots & 0 \\
\vdots & \vdots & \vdots & \vdots & \vdots & \ddots & \vdots & \vdots\\
0 & \cdots & \cdots & \cdots & \cdots & \cdots & \cdots &\Gamma_{m,k_m}
\end{pmatrix}
\end{align*}

\noindent and the map $p$ is the projection that sends the first $k_1$ coordinates of $\NN^{m''}$ to the first coordinate of $\NN^{m}$, the next $k_2$ coordinates to the second coordinate of $\NN^{m}$, and so forth. It remains to explain how the morphism lifts from the level of characteristic monoids to the actual log structures; this is the evident extension of formula \ref{eq:1'} of remark 2. If $s^{C'/S'}$ mapped

\begin{align*}
e_{ij} \mapsto (\Gamma_{ij}e_i,u_{ij})
\end{align*}

\noindent we now have that $r^{C'/S'}$ maps

\begin{align*}
e_{ij} \mapsto (\Gamma_{ij}e_{ij},u_{ij})
\end{align*}

\noindent Therefore, from the data $(C' \rightarrow S', N^{C'/S'} \rightarrow N)$ we obtain a log twisted curve $(C' \rightarrow S', r^{C'/S'}:N^{C'/S'} \rightarrow N^{C'/S'})$. This defines the required morphism. When we compose with the isomorphism $\mathfrak{M}_{g,n}^{\textup{logtw}} \cong \mathfrak{M}_{g,n}^{\textup{tw}}$, the twisted curve $\mathcal{C}$ we obtain is the curve $C$ with the $j$-th node of $C$ mapping to the $i$-th node of $W$ twisted by $\Gamma_{ij}$.  \\

Next, \'{e}tale locally around the image of $f$ in $\mathfrak{M}_{g,n}^{\textup{logtw}}$, we have a morphism $\mathfrak{M}_{g,n}^{\textup{logtw}} \rightarrow \mathcal{A}^{m''}$. \'{e}tale locally around $\Spec k$ in $S$, the $m''$ nodes of $C$ determine $m''$ divisors $D_i$, the locus of points in $S$ over which the node persists. A divisor determines a line bundle with a section, hence, the divisors $D_i$ determine an element of $\mathcal{A}^{m''}$. We may describe this map alternatively as follows: The $m''$ nodes signify that \'{e}tale locally the tautological map $S \rightarrow \mathfrak{M}_{g,n}$ maps $C$ to the intersection of $m''$ divisors in $\mathfrak{M}_{g,n}$ intersecting normally, which corresponds to $m''$ line bundles with sections in $\mathfrak{M}_{g,n}$. We may pull these line bundles with sections to $S$, to obtain the desired element of $\mathcal{A}^{m''}$. We may further obtain a morphism $\mathcal{B}^{etw} \rightarrow \mathcal{A}^{m}$ in a similar fashion. \\

Putting everything together, we obtain an \'{e}tale neighborhood $\mathcal{MB}_f$ of $\tau(f)$ in $\mathcal{MB}$ and a morphism $\phi$
  
\begin{align*}
\xymatrix{\ar @{} [dr] |{}
\phi: \mathcal{MB}_f \ar[r] & \mathfrak{M}^{\textup{logtw}} \times_{\mathcal{A}^{m''}}\mathcal{A}^{m}  \ar[d] \ar[r] & \mathcal{A}^{m} \ar[d] \\
 & \mathfrak{M}^{\textup{logtw}} \ar[r] & \mathcal{A}^{m''} }
\end{align*}

\begin{lem}
The morphism $\phi$ is \'{e}tale. 
\end{lem}

\begin{proof} Since all stacks in question are smooth, it is sufficient to show that their tangent spaces at $\tau(f)$ are isomorphic. It will be clear from the proof that we may reduce to the case where two distinguished nodes map into a single node of the target, that is, where $m''=2,m=1,m'=0$. The argument for this is the same as the argument in the proof of lemma $2$.  We will restrict attention to this case to simplify the notation. We are therefore given \'{e}tale locally around $\tau(f)$ a diagram 

\begin{align*}
\xymatrix{\ar @{} [dr] |{}
\phi: \mathcal{MB} \ar[r] & \mathfrak{M}^{\textup{logtw}} \times_{\mathcal{A}^{2}}\mathcal{A}^{1}  \ar[d] \ar[r] & \mathcal{A}^{1} \ar[d] \\
 & \mathfrak{M}^{\textup{logtw}} \ar[r] & \mathcal{A}^{2} }
\end{align*}

\noindent and want to show that $\phi$ induces an isomorphism of tangent spaces. Recall that the element $\tau(f) \in \mathcal{MB}(\Spec(k))$ consists of data of a pair $(C/\Spec k, W=X[1]/\Spec[k])$, and two diagrams of log structures 

\begin{align}
\xymatrix{\ar @{} [ldr] |{}
\NN^2 \oplus k^{*} \ar[r]^{s} \ar[rd] & \NN \oplus k^{*} \ar[d]^{\kappa} \\
&k  }, \xymatrix{\ar @{} [ldr] |{}
\NN \oplus k^{*} \ar[r]^{t} \ar[rd] & \NN \oplus k^{*} \ar[d]^{\kappa} \\
&k  }
\end{align} 

\noindent Here the maps $s$ and $t$ are the maps $s^{C/\Spec k}, t^{W/\Spec k}$ respectively. We have by formula \ref{eq:1'} that $e_i \mapsto (\Gamma_ie,u_i)$ (The $\Gamma_i$ here are what we would have called $\Gamma_{1i}$ above, but since there is only one target node, we drop the first index to simplify notation). The map $t$ is given by mapping the generator $e \rightarrow (de,1)$, and the rest of the arrows send the generators of $\NN,\NN^2$ to $0$ in $k$. \\

On the other hand, an element of $\mathfrak{M}^{\textup{logtw}} \times_{\mathcal{A}^{2}}\mathcal{A}^{1}(\Spec k)$ corresponds to a triple $(x,y,\alpha)$ of an element $x$ of $\mathfrak{M}^{\textup{logtw}}(\Spec k)$, an element $y \in \mathcal{A}^{1}(\Spec k)$, and an isomorphism between their images in $\mathcal{A}^{2}(\Spec k)$. The element $x$ corresponds to a pair of a nodal curve $C/\Spec k$ as above and a diagram of log structures  

\begin{align}
\xymatrix{\ar @{} [ldr] |{}
N^{C/k} = \NN^2 \oplus k^{*} \ar[r]^-{r} \ar[rd] & \NN^2 \oplus k^{*} \ar[d]^{\lambda} \\
&k  }
\end{align}

\noindent where the top map $r = r^{C/\Spec k}$ is an injection $e_i \mapsto (\Gamma_ie_i,u_i)$; the  map $\lambda$ sends $e_i \mapsto a_i \in k$; and the diagonal map is determined by commutativity, $e_i \mapsto u_ia_i^{\Gamma_i}$. An element of $\mathcal{A}^{1}(\Spec k)$ is a line bundle over $\Spec k$ together with a section, in other words, an element $a \in \Aff^1$. An element of $\mathcal{A}^{2}(\Spec k)$ is similarily a pair $(a_1,a_2) \in \Aff^2$. The map $\mathcal{A}^1(\Spec k) \rightarrow \mathcal{A}^2(\Spec k)$ sends $a \mapsto (a,a)$ and the map $\mathfrak{M}^{\textup{logtw}}(\Spec k) \rightarrow \mathcal{A}^2(\Spec k)$ sends the data just described to the pair $(a_1,a_2)$. Therefore the triple $(x,y,\alpha)$ has $x$ as above, $y=a \in \Aff^1$, and $\alpha = (c_1,c_2) \in (k^{*})^2$ an isomorphism of $(a,a)$ with $(a_1,a_2)$ in $\Aff^2$, that is 

\begin{itemize}
\item $c_i = \frac{a_i}{a}$ if all $a$ and $a_i$ are non-zero 
\item $c_i$ arbitrary if $a=a_1=a_2=0$. 
\end{itemize}

\noindent The morphism $\phi: \mathcal{MB} \rightarrow \mathfrak{M}^{\textup{logtw}} \times_{\mathcal{A}^{2}}\mathcal{A}^{1}$ then sends the data corresponding to $\tau(f)$ to the triple $(x,0,(1,1))$, where $x$ is the curve $(C/\Spec k)$ and the diagram (3) has $r,\beta$ determined by $r(e_i) = (\Gamma_ie_i,u_i)$ and  $\beta(e_i)= 0$, with $\Gamma_i,u_i$ as in the definition of $s$. \\

To show that $\phi$ induces an isomorphism of tangent spaces we consider isomorphism classes of morphisms from $\Spec k[\epsilon]$ to all stacks in question extending the given data over $\Spec k$. A morphism $\Spec k[\epsilon] \rightarrow \mathfrak{M}^{\textup{logtw}}$ corresponds to a pair of an infinitesimal deformation $C'/\Spec k[\epsilon]$ of $C$ and a diagram 

\begin{align}
\xymatrix{\ar @{} [ldr] |{}
 \NN^2 \oplus k[\epsilon]^{*} \ar[r]^{r[\epsilon]} \ar[rd] & \NN^2 \oplus k[\epsilon]^{*} \ar[d]^{\lambda[\epsilon]} \\
&k[\epsilon]  }
\end{align}

\noindent lying over (3). Therefore, we must have $e_i \rightarrow (\Gamma_ie_i,u_i+v_i\epsilon)$ under $r[\epsilon]$; $\lambda[\epsilon]$ maps $e_i \rightarrow \alpha_i \epsilon$. The diagonal arrow is determined by commutativity $e_i \mapsto (\alpha_i)^{\Gamma_i}(u_i+v_i\epsilon)$. \\

Morphisms $\Spec k[\epsilon] \rightarrow \mathcal{A}^1$ and $\Spec k[\epsilon] \rightarrow \mathcal{A}^2$ lying over the given elements $0 \in \mathcal{A}^{1}(\Spec k), (0,0) \in \mathcal{A}^2(\Spec k)$ are again line bundles over $\Spec k[\epsilon]$, which are thus trivial, together with a section restricting to 0 over $\Spec k$; hence they correspond to $\alpha \epsilon, (\alpha_1\epsilon,\alpha_2\epsilon)$ respectively. Under the morphism $\mathfrak{M}^{\textup{logtw}}(k[\epsilon]) \rightarrow \mathcal{A}^{2}(\Spec k[\epsilon])$ maps the extension of $x$ to the pair $(\alpha_1\epsilon,\alpha_2\epsilon)$, where the $\alpha_i$ are the ones appearing in diagram (4). Isomorphisms between $(\alpha \epsilon, \alpha \epsilon)$ and $(\alpha_1\epsilon,\alpha_2\epsilon)$ restricting to $(1,1)$ over $\Spec k$ are of the form $(1+\beta_1\epsilon,1+\beta_2\epsilon)$. Note that $\beta_1,\beta_2$ can be arbitrary; however, in order for an isomorphism to exist, the condition $\alpha_1=\alpha_2=\alpha$ is forced. Therefore, the choices involved in extending $(x,y,\alpha)$ are the choices of the deformation $C'$ of $C$ and the numbers $\alpha,v_i,\beta_i$. Notice however that the choice of either the $v_i$ or the $\beta_i$ can be eliminated via an isomorphism. For consider two pairs $(1+\beta_1\epsilon,1+\beta_2\epsilon)$ and $(1+\beta_1'\epsilon,1+\beta_2'\epsilon))$. There is an isomorphism 

\begin{align*}
\xymatrix{\ar @{} [ldr] |{}
 \NN^2 \oplus k[\epsilon]^{*}  \ar[d]_{=} &\stackrel{r[\epsilon]:e_i \mapsto (\Gamma_ie_i,u_i+v_i\epsilon)}{\xrightarrow{\hspace*{3cm}}} & \NN^2 \oplus k[\epsilon]^{*} \ar[d]^{A} \ar[r]^-{\lambda[\epsilon]} & k[\epsilon] \ar[d]^{=}\\
 \NN^2 \oplus k[\epsilon]^{*}  &\stackrel{r'[\epsilon]:e_i \mapsto (\Gamma_ie_i,u_i+v'_i\epsilon)}{\xrightarrow{\hspace*{3cm}}} & \NN^2 \oplus k[\epsilon]^{*} \ar[r]_-{\lambda[\epsilon]} &k[\epsilon]  }
\end{align*}

\noindent with the vertical arrow being the isomorphism $A:e_i \mapsto (e_i,1+c_i\epsilon)$ where $c_i=\beta_i - \beta_i'$ and $v_i' = v_i + u_i\Gamma_ic_i$. In other words, if we denote by $x[\epsilon]$ the extension of $x$ determined by choosing $r[\epsilon]$ as the extension of $r$ (that is, with the choice of the unit $v_i$) and by $x[\epsilon]'$ the one with $r'[\epsilon]$ as the extension of $r$ (with the unit $v_i'$) we have an isomorphism between the triple $(x[\epsilon],\alpha\epsilon, (1+\beta_1\epsilon,1+\beta_2\epsilon))$ and $(x[\epsilon]',\alpha\epsilon, (1+\beta_1'\epsilon,1+\beta_2'\epsilon))$. \\

To summarize, the choice of the extension of the image of $\tau(f)$ in $\mathfrak{M}^{\textup{logtw}}$ corresponds to the data of a choice of a deformation $C'$ of $C$, and the choices of the numbers $\alpha,v_i$. \\

On the other hand, a choice of an extension of $\tau(f)$ in $\mathcal{MB}$ corresponds to the data of a deformation $C'$ of $C$, a deformation $W'$ of $W$ in $\mathcal{B}$, which is necessarily trivial, and diagrams  

\begin{align}
\xymatrix{\ar @{} [ldr] |{}
\NN^2 \oplus k[\epsilon]^{*} \ar[r]^{s[\epsilon]} \ar[rd] & \NN \oplus k[\epsilon]^{*} \ar[d]^{\kappa[\epsilon]} \\
&k[\epsilon]  }, \xymatrix{\ar @{} [ldr] |{}
\NN \oplus k[\epsilon]^{*} \ar[r]^{t[\epsilon]} \ar[rd] & \NN \oplus k[\epsilon]^{*} \ar[d]^{\kappa[\epsilon]} \\
&k[\epsilon]  }
\end{align} 

\noindent lying over (2). The extension $s[\epsilon]$ maps $e_i \mapsto (\Gamma_ie,u_i+v_i\epsilon)$, $t[\epsilon]$ maps $e \mapsto (de,1+v\epsilon)$, $\kappa[\epsilon]$ maps $e \mapsto \alpha \epsilon$ and the diagonal arrows are determined by commutativity. Notice that up to isomorphism, there is only one choice for the right diagram, the choice of the number $\alpha$. Again, this is because from the diagram

\begin{align*}
\xymatrix{\ar @{} [ldr] |{}
 \NN \oplus k[\epsilon]^{*}  \ar[d]_{=} &\stackrel{t[\epsilon]:e \mapsto (de_i,1+v\epsilon)}{\xrightarrow{\hspace*{3cm}}} & \NN \oplus k[\epsilon]^{*} \ar[d]^{B} \ar[r]^-{\kappa[\epsilon]} & k[\epsilon] \ar[d]^{=}\\
 \NN \oplus k[\epsilon]^{*}  &\stackrel{t'[\epsilon]:e \mapsto (de,1+v'\epsilon)}{\xrightarrow{\hspace*{3cm}}} & \NN \oplus k[\epsilon]^{*} \ar[r]_-{\kappa[\epsilon]} &k[\epsilon]  }
\end{align*}

\noindent with $B$ the isomorphism $e \mapsto (e,1+c\epsilon)$, and $v' =v+ cd\epsilon$ we get an isomorphism between the extension determined by the extensions $s[\epsilon],\kappa[\epsilon]$ of $s,\kappa$ with the extension determined by $s'[\epsilon],\kappa[\epsilon]$. Since by varying $c$ the expression $v+cd\epsilon$ varies through all elements of $k$, the choice of $s[\epsilon]$ is eliminated up to isomorphism, as claimed. Once the isomorphism $B$ is fixed, though, the diagrams on the left with different choices of $v_i$ remain distinct. In other words, the choices involved in extending $\tau(f)$ are up to isomorphism the extension $C'$ of $C$ and the numbers $\alpha,v_i$. These are precisely the same choices as involved in extending $(x,0,(1,1))$. This concludes the proof of the lemma.
\end{proof}

\begin{remark} The geometric meaning of the number $\alpha_i$ in the map $\NN^2 \oplus k[\epsilon]^{*} \rightarrow k[\epsilon], e_i \mapsto \alpha_i\epsilon$ is that the $i$-th node is smoothed with speed $\alpha_i$ in the moduli space of twisted curves. The geometric significance of lemma 4 then is that in $\mathcal{MB}$, all nodes mapping to the same node of the target must be smoothed simultaneously, with the same speed: the speed with which the node of the target is being smoothed. 
\end{remark}

The lemma in particular implies that we may calculate $\textrm{Def}(\tau(f))$ as follows: Let us write $\sum x_i$ for the divisor of marked points and $\mathcal{C}$ the twisted curve obtained from $f$ as explained. The tangent space to the stack of twisted curves is given by the ext group

\begin{align*}
Ext^{1}((\Omega_{\mathcal{C}}(\sum{x_i}),\mathcal{O}_\mathcal{C})
\end{align*}

\noindent where $\OO_{\mathcal{C}}$ and $\Omega_{\mathcal{C}}$ are the structure sheaf and sheaf of Kahler differentials of the twisted curve, respectively. The ''local-to-global'' spectral sequence for $Ext$ says that the the tangent space fits into the short exact sequence

{\footnotesize
\begin{align}
0 \rightarrow H^{1}(\und{Hom}(\Omega_{\mathcal{C}}(\sum x_i),\OO_{\mathcal{C}}) \rightarrow Ext^{1}((\Omega_{\mathcal{C}}(\sum x_i),\mathcal{O}_\mathcal{C}) \rightarrow H^{0} (\und{Ext^1}(\Omega_{\mathcal{C}}(\sum x_i),\OO_{\mathcal{C}})) \rightarrow 0
\label{(6)}
\end{align}
}%

\noindent Here Hom and Ext are underlined to indicate that we are taking the sheaf Hom and Ext respectively. The rightmost group in the exact sequence has a canonical description as follows. Let $R,S,T$ denote the set of non-distinguished nodes of the curve, the set of distinguished nodes of the curve, and the set of nodes of the target, as above. Furthermore, given a node $x$ in the curve $\mathcal{C}$, let $\mathcal{C}^i_x, i=1,2$ denote the two components of $\mathcal{C}$ at $x$. Then we have

\begin{align*}
H^{0} (\und{Ext^1}(\Omega_{\mathcal{C}}(\sum x_i),\OO_{\mathcal{C}})) \cong \bigoplus_{\textup{nodes of C}} T_x\mathcal{C}^1_x \tensor T_x\mathcal{C}^2_x \cong \CC^{m'+m''}
\end{align*}

There is a diagonal map $\CC^{m} \rightarrow \CC^{m''}$, which simply sends the coordinate $e_y$ cooresponding to a node $y \in T$ to the sum of the coordinates corresponding to the nodes in $\mathcal{C}$ mapping to $y$, that is, $\sum_{x;f(x)=y} e_x$. This is in fact the map of tangent spaces of the map $\mathcal{A}^{m} \rightarrow \mathcal{A}^{m''}$ described above. Just as the tensor product $\bigoplus_{\textup{nodes of C}} T_x\mathcal{C}^1_x \tensor T_x\mathcal{C}^2_x$ describes intrinsically the part of the deformations of the curve that smooth the nodes, the group $\oplus_{i=1}^nH^0(\mathcal{W},N_{\mathcal{D}[i]/\mathcal{X}[i-1]}\tensor N_{\mathcal{D}[i]/\mathcal{X}[i]})$ describes the part of the deformations of $\mathcal{W}=\mathcal{X}[n]$ that smooth the nodes of $\mathcal{W}$. Here, $\mathcal{W}$ is obtained from $W$ by twisting along the divisor $D[i]$ by the integer $d_i$ via the map $t^{W/S}:N^{W/S} \rightarrow N$, just as $\mathcal{C}$ is obtained from $C$ via the map $r^{C/S}$. Then, the fiber diagram of lemma 4 implies:

\begin{cor}\label{cor1} The tangent space $\textrm{Def}(\tau(f))$ to $\mathcal{MB}$ is the fiber product

\begin{align*}
\xymatrix{\ar @{} [dr] |{}
\textrm{Def}(\tau(f)) \ar[d] \ar[r] & \CC^{m} \cong \bigoplus_{i=1}^{n} H^{0}(N_{\mathcal{D}[i]/\mathcal{X}[i-1]}\tensor N_{\mathcal{D}[i]/\mathcal{X}[i]}) \ar[d] \\
Ext^{1}((\Omega_{\mathcal{C}}(\sum x_i),\mathcal{O}_\mathcal{C}) \ar[r] & \CC^{m''} \cong \bigoplus_{\textup{distinguished nodes of C}} T_x\mathcal{C}^1_x \tensor T_x\mathcal{C}^2_x }
\end{align*}
\end{cor}

The proof of lemma 5 also allows us to understand the automorphism group $\mathfrak{aut}(\tau(f)):$

\begin{cor}
An infinitesimal automorphism of $f$ in $\mathcal{MB}$ consists of an infinitesimal automorphism of the source curve $C$ fixing the marked points and an infinitesimal automorphism of $W$ in $\mathcal{B}$: 
\begin{align*}
\mathfrak{aut}(\tau(f)) = \mathfrak{aut}(C,\vec{x}) \oplus \mathfrak{aut}_\mathcal{B}(W) = \mathfrak{aut}(C,\vec{x}) \oplus \CC^{m}
\end{align*}
\end{cor}

\begin{proof} We give the details for the case when two nodes of $C$ map into a node of $W$, as in lemma 5. The general case reduces to this as in the proof of lemma 2. The group of infinitesimal automorphisms of $\tau(f)$ consists of the group of automorphism of the trivial extension of $\tau(f)$ over $k[\epsilon]$ lying over the identity automorphism of $\tau(f)$. This is an infinitesimal automorphism of $C$, an infinitesimal automorphism of $W$ in $\mathcal{B}$, which consists of a copy of $\CC$ for each expanded component of $W$, and an automorphism $A$ of the log structure $N = \NN \oplus k[\epsilon]^*$ which lies over the identity automorphism over $\Spec k$ and which is compatible with both diagrams 

\begin{align*}
\xymatrix{\ar @{} [ldr] |{}
N^{C/S} = \NN^2 \oplus k[\epsilon]^{*} \ar[r]^-{s} \ar[rd] & N = \NN \oplus k[\epsilon]^{*} \ar[d] \\
&k[\epsilon]  }, \xymatrix{\ar @{} [ldr] |{}
N^{W/S} = \NN \oplus k[\epsilon]^{*} \ar[r]^-{t} \ar[rd] & N = \NN \oplus k[\epsilon]^{*} \ar[d] \\
&k[\epsilon]  }
\end{align*}

\noindent Here $s(e_i)= (\Gamma_ie,u_i)$, $t(e) = (de,1)$ as always. An automorphism $A$ of $N$ must map the generator of $\NN$ $e \mapsto (e,(1 + c\epsilon))$ for some $c \in k$ if it is to reduce to the identity over $\Spec k$.  In order for the automorphism to be compatible with the second diagram, that is, in order for

\begin{align*}
\xymatrix{\ar @{} [ldr] |{}
\NN \oplus k[\epsilon]^{*} \ar[r]^t \ar[d]_{=} & \NN \oplus k[\epsilon]^{*} \ar[d]^{A} \\
\NN \oplus k[\epsilon]^{*} \ar[r]^t& \NN \oplus k[\epsilon]^{*}},  
\end{align*}

\noindent to commute, we must have $At(e) = (de, (1+c\epsilon)^d)=(de, 1+cd\epsilon) = (de,1) = t(e)$, which implies that $c=0$. So the logarithmic structures contribute no infinitesimal automorphisms. 
\end{proof} 

\end{section}

\begin{section}{The Virtual Localization Formula}
We are now in a position to derive the virtual localization formula for $\Kim$, in the case when the pair $(X,D)$ carries a $T=\CC^{*}$-action leaving $D$ pointwise fixed. The ideas of this section can essentialy be found in the paper of Graber-Vakil \cite{GV}. For the convenience of the reader, we recall the form of relative virtual localization formulas and refer the reader to the paper of Graber-Pandharipande \cite{GP} for details. \\

\begin{subsection}{Graber-Pandharipande Virtual Localization} Suppose $M$ is a DM stack with a $\CC^{*}$-action equipped with a $\CC^{*}$-equivariant perfect obstruction theory $C^{\bullet} \rightarrow \mathcal{L}_{M}$. Let $F_a$ denote the connected components of the fixed locus of $M$, which we refer to as the fixed loci for brevity, and denote the natural inclusion by $i_a:F_a \rightarrow M$. The relative virtual localization formula reads    

\begin{align*}
\int_{[M]^{\textup{vir}}}{\omega} = \sum_{\alpha} \int_{[F_{\alpha}]^{\textup{vir}}}{\frac{i_{\alpha}^{*}\omega}{e^T(N_F^{\textup{vir}})}}
\end{align*}

\noindent Here $\omega$ is a class in $A_T^{*}(M)$, the equivariant Chow ring of $M$ (or in equivariant cohomology). The integral $\int_{[M]^{\textup{vir}}}$ is the proper pushforward map from $A_T^{*}(M) \rightarrow A_T^{*}(pt)=A^{*}(BT)\cong \CC[u]$. The term $e^{T}$ denotes the equivariant Euler class of a vector bundle, in this case, of the virtual normal bundle $N_{F_{\alpha}}^{\textup{vir}}$ of $F_{\alpha}$ in $M$. The qualification that the normal bundle is virtual means that we are not only taking the ordinary normal bundle in the tangent space, but rather that also keeping track of the obstruction bundle. More precisely, $N_F^{\textrm{vir}}$ is defined as $\mathcal{T}_1^m - \mathcal{T}_2^m$, the moving part of the tangent space minus the moving part of the obstruction space (the moving part of a representation is the subrepresentation where $T$ acts non-trivially). The Euler class of a sum of vector bundles is by definition the product of the Euler classes; the Euler class of the difference is thus the quotient. Finally, the virtual fundamental class of a fixed locus $[F]^{\textrm{vir}}$ is by definition the virtual fundamental class arising from the fixed part of the tangent/obstruction theory: $\mathcal{T}_1^f - \mathcal{T}_2^f$. Therefore, in order to give a localization formula we must identify the fixed loci and calculate the classes $e^T(N_F^{\textup{vir}})$ for each of them. \\

\end{subsection}

\begin{subsection}{Types of Fixed Loci}
Consider a pair $(X,D)$ as above. This defines a stack of expanded targets $\mathcal{B}$, as explained in remark x. We fix discrete data $\Gamma = (g, m, h, \vec{c}=(c_1,\cdots,c_h))$, consisting of the genus of a curve, $m+h$ marked points, and the contact orders of the last $h$ marked points at the divisor at infinity $D[n]$ of the target $W=X[n]$. In this section we study the fixed loci of $\Kim = M$. In the localization formula we will distinguish between two different types of fixed loci; the first type consists of morphisms to $X$ itself, rather than an expansion of $X$. Such a locus is much simpler to understand than a general locus. The other type of fixed loci consist of those with targets $W=X[n]$ with $n > 0$. The idea of the localization formula is to express the virtual fundamental classes of these loci recursively, in terms of the simple loci and moduli spaces of log stable maps to the expanded part of $W$ only. We formalize this below.

\begin{definition}
A fixed locus $F \subset M$ is called \emph{simple} if the general, hence every, element $f \in F$ has target $W = X[0] = X$. We denote a simple fixed locus with discrete data $\Gamma$ by $\Kim^{\textup{sim}} = M^{\textup{sim}}$.  
\end{definition}

\noindent Similarily we define composite loci: 

\begin{definition}
A fixed locus $F \subset M$ is called \emph{composite} if the elements $f \in F$ map to targets $W=X[n]$ with $n >0$. 
\end{definition}

The simple loci $\Kim^{\textup{sim}}$ are open substacks of $\Kim$; therefore, all results of section 4 apply to the simple loci without change. \\

To understand composite loci, we need to understand log stable maps to the expanded part of $W = X[n]$, that is, to $Y = \cup_{i=1}^{n} X[i]$. The formal definition of such a log stable map is identical to the one given in Definition 5. The only differences in the theory of such maps arise from the fact that the rigidifying map to $X$ is much simpler, contracting all components to $D$. In the case of $X = \PP^1$, which was the first case to be studied in the literature, the rigidifying map to $X$ becomes trivial. We thus call these maps unrigidified log stable maps. We denote the stack parametrizing the expanded part of expansions of $X$ by $\mathcal{B}^{\sim}$ and its universal family by $\mathcal{U}^{\sim}$; similarily to $\mathcal{B}$, the stack $\mathcal{B}^{\sim}$ is isomorphic to the open substack of $\mathfrak{M}_{0,2}$ where the first marking is on the first component and the second marking is on the last component - see remark 2. We will abusively denote the space of log stable maps to targets in $\mathcal{B}^{\sim}$ by $\Kim^{\sim}$. The analogous stack of unrigidified relative stable maps $\Li^{\sim}$ is introduced and studied in \cite{GV}. 

\begin{remark} The stack $\Kim^{\sim}$ is very similar and often simpler than $\Kim$. For instance, observe that 
\begin{itemize}
\item The minimality condition in Definition 5 is actually simpler: it simply requires that the log curve $(C,M) \rightarrow (S,N)$ is minimal, since it is not possible to have a node of the target with no distinguished node mapping to it anymore. \\
\item The analogous map $\Kim^{\sim} \rightarrow \Li^{\sim}$ of lemma 2 is also finite, as the map to $X$ is never required in the proof of lemma 2. \\
\item Theorem 1 applies without change. This is because it is shown in \cite{GV} that $\Li^{\sim}$ satisfies the property SE, so Lemma 1 applies. \\
\end{itemize}
\end{remark}

For the purposes of localization, we need to understand the deformation theory of $\Kim^{\sim}$ carefully. This is essentially the same as that of $\Kim$. Specifically, introduce the analogue of the stack $\mathcal{MB}$, 

\begin{align*}
\mathcal{MB}^{\sim} = \mathfrak{M}_{g,n}^{\textup{log}} \times_{\textup{LOG}} (\mathcal{B}^{\sim})^{\textup{etw}}
\end{align*}

\noindent and its canonical morphism 

\begin{align*}
\tau^{\sim}: \Kim^{\sim} \rightarrow \mathcal{MB}^{\sim} 
\end{align*}

\noindent The analogues of lemmas 3,5 and of corollaries 1,2 are then as follows:

\begin{lem} Over a geometric point $f:(C,M_C)/(k,N) \rightarrow (Y,M_Y)/(k,N)$, the tangent space $\mathcal{T}^1$ and obstruction space $\mathcal{T}^2$ of $M^{\sim}$ fit into an exact sequence
\begin{align*}
0 \rightarrow \mathfrak{aut}(\tau^{\sim}(f)) \rightarrow H^{0} (f^{*}T^{\log}_Y) \rightarrow \mathcal{T}^1 \rightarrow \textrm{Def}(\tau^{\sim}(f)) \rightarrow H^{1}(f^{*}T^{\log}_Y) \rightarrow \mathcal{T}^2 \rightarrow 0.
\end{align*}
\end{lem}

\begin{proof} The deformation theory of $\Kim^{\sim}$ is the same as in section 4, induced by the relative perfect obstruction theory

\begin{align*}
(E^{\sim})^{\bullet} = (R\pi_{*}f^{*}T_{\mathcal{U^{\sim}}^{\textrm{etw}}/\mathcal{B^{\sim}}^{\textrm{etw}}}^{\textup{log}})^{\vee} \rightarrow \mathcal{L}^{\log}_{M^{\sim}/\mathcal{MB^{\sim}}}
\end{align*}

\noindent where $M$ has been replaced by $M^{\sim}$ and $\mathcal{B}$ by $\mathcal{B}^{\sim}$. This is true since Kim's proof works for any stack of Fulton-Macpherson spaces, thus, expansions and unrigidified expansions alike. Therefore, the discussion of section 4 applies as well, which results in the six term exact sequence. \end{proof}

\noindent Furthermore, with the notations of corollary \ref{cor1}, we have:  

\begin{cor} The tangent space $\textrm{Def}(\tau^{\sim}(f))$ to $\mathcal{MB}$ is the fiber product

\begin{align*}
\xymatrix{\ar @{} [dr] |{}
\textrm{Def}(\tau^{\sim}(f)) \ar[d] \ar[r] & \CC^{m} \cong \bigoplus_{i=2}^{n} H^{0}(N_{\mathcal{D}[i]/\mathcal{X}[i-1]}\tensor N_{\mathcal{D}[i]/\mathcal{X}[i]}) \ar[d] \\
Ext^{1}((\Omega_{\mathcal{C}}(\sum x_i),\mathcal{O}_\mathcal{C}) \ar[r] & \CC^{m''} \cong \bigoplus_{\textup{distinguished nodes of C}} T_x\mathcal{C}^1_x \tensor T_x\mathcal{C}^2_x }
\end{align*}

\end{cor}

\begin{proof} The corollary follows from lemma 5; the lemma applies verbatim, as the map to $X$ is required nowhere in the proof.  
\end{proof}

Finally, 

\begin{cor}
An infinitesimal automorphism of $f$ in $\mathcal{MB}^{\sim}$ consists of an infinitesimal automorphism of the source curve $C$ fixing the marked points and an automorphism of $W$ in $\mathcal{B}^{\sim}$: 
\begin{align*}
\mathfrak{aut}(\tau(f)) = \mathfrak{aut}(C,\vec{x}) \oplus \mathfrak{aut}_{\mathcal{B}^{\sim}}(W) = \mathfrak{aut}(C,\vec{x}) \oplus \CC^{m}
\end{align*}
\end{cor}

The two main differences between $\Kim$ and $\Kim^{\sim}$ are the following: First, the action of $T=\CC^*$ on $\Kim^{\sim}$ is trivial. This is because dilation of each component of $Y= \cup_{i=1}^{n} X[i]$ by $c \in T$ is an automorphism of $Y$, and so the map $cf$ is isomorphic to $f$. Second, the log tangent bundle $T_Y^{\textup{log}}$ is trivial. \\

\end{subsection}

\begin{subsection}{Description of the fixed loci in terms of known stacks} 

Consider now an element $f \in \Kim$ in a composite fixed locus, mapping to a target $W=X[n]$. Define $C_1 = f^{-1}(X), C_2 = f^{-1}(Y)$ with $Y = \cup_{i=1}^n X[i]$ as above, and let $f_1:C_1 \rightarrow X, f_2:C_2 \rightarrow Y$ denote the restriction maps. Furthermore, $f_1$ and $f_2$ are naturally logarithmic maps, from logarithmic data determined only from $f$. Specifically, recall from definition \ref{def5} and remark \ref{rmk1} following it, that the additional information on the underlying map $\und{f}$ that make $f$ into a logarithmic map are: (1) A locally free sheaf of monoids $N = N_{q} \oplus N_{D_0}\oplus N_{D_1} \oplus \cdots N_{D_n}$, where $N_q$ is the locally free log structure with one factor of $\NN$ for each non-distinguished node of $C$, and $N_{D_i}$ is a rank 1 locally free log structure corresponding to the $i$-th node of $W$, and (2) homomorphisms 

\begin{align*}
N^{C/S} \rightarrow N, N^{W/S} \rightarrow N
\end{align*} 

\noindent We define a logarithmic structure $N_1$ on $S$ by taking the sub-log structure of $N$ corresponding to the non-distinguished nodes of $C_1$, together with the natural restriction homomorphisms 

\begin{align*}
N^{C_1/S} \rightarrow N_1, N^{X/S}=0 \rightarrow N_1
\end{align*}

\noindent This data turns $f_1$ into a logarithmic map. Similarily, define a log structure $N_2$ which is the sub-log structure of $N$ corresponding to the non-distinguished nodes of $C_2$ and the nodes $D[1],D[2],\cdots,D[n]$ of $Y$, together with the restriction homomorphisms

\begin{align*}
N^{C_2/S} \rightarrow N_2, N^{Y/S} \rightarrow N_2
\end{align*}

\noindent The logarithmic data from $f_1$ and $f_2$ are in fact almost equivalent to the logarithmic data of $f$. First, $N$ is determined by $N_1$ and $N_2$, as it is $N_1 \oplus N_2 \oplus N_{D_0}$, where $N_{D_0}$ is the locally free log structure corresponding to deformation of the node $D=D[0]$. The homomorphisms $N^{C/S} \rightarrow N, N^{W/S} \rightarrow N$ are determined from those of $N_1,N_2$, except for the part of the map $N^{C/S} \rightarrow N_{D_0}$ corresponding to the $k$ nodes over $D$. The only additional data $f$ carries is thus a homomorphism $\NN^k \rightarrow N_{D_0} = \NN \oplus O^{*}_{S}$, which corresponds to a collection of units $u_1,\cdots,u_k$ - see equation \ref{eq:1'}. Here $u_i$ is an $\alpha_i$-th root of unit, where $\alpha_i$ is the degree of $f$ at the $i$-th node over $D$. Note that on a composite fixed locus the node $D[0]$ has to persist. Furthermore, since $\mu_{\alpha_i}$ is a discrete group, and the units $u_i$ must vary continuously, the $u_i$ determined by $f$ are constant on the connected fixed locus containing $f$. Recall that by remark 6.3.1 of \cite{Kim}, the units $\vec{u}$ determine the same log map if they differ by multiplication by an element of $\mu_d$, where $d$ is the lcm of the $\alpha_i$. Briefly, then, the logarithmic data of $f$ is determined by the logarithmic data of $f_1$ and $f_2$ and the collection $\vec{u}=(u_1,\cdots,u_k) \in \prod \mu_{\alpha_i}/\mu_{d}$.  \\  

The discrete data $\Gamma = (g,h,m,\vec{c})$ also splits into two sets of discrete data: $\Gamma_1=(g_1, \vec{\alpha} = (\alpha_1,\cdots,\alpha_k), S_1, \beta_1)$, consisting of the genus of $C_1$, a partition describing the behavior of $f_1$ over $D=D[0]$, the subset $S_1$ of marked points on $C_1$, and the homology class $(f_1)_*[C_1]$; and $\Gamma_2 = (g_2, \vec{\alpha}, \vec{c}, S_2, \beta_2)$, consisting of the genus of $C_2$, the same partition $\vec{\alpha}$, the original partition $\vec{c}$ describing the behavior along the divisor at infinity, and the homology class determined by $f_2$. Then, the map $f_1$ belongs to $\Kims{\Gamma_1}$ and $f_2$ belongs to $\Kimc{\Gamma_2}$. \\

Summarizing, the data of $f$ is equivalent to the data $(f_1,f_2,\vec{u})$ of the pair of the maps and the units, which can be glued to a log stable map $f$ to $W$, provided that $f_1$ and $f_2$ map the marked points of $C_1$ and $C_2$ to the same points of the divisor $D$. In other words, if we denote by $F_{\Gamma_1,\Gamma_2,\vec{u}}$ the fixed locus for which $f_1 \in \Kims{\Gamma_1}$, $f_2 \in \Kimc{\Gamma_2}$, we get a map 

\begin{align*}
\Kims{\Gamma_1} \times_{D^k} \Kimc{\Gamma_2} \times \prod{\mu_{\alpha_i}/\mu_d}  \rightarrow F_{\Gamma_1,\Gamma_2,\vec{u}} 
\end{align*}

\noindent This map identifies $F_{\Gamma_1,\Gamma_2,\vec{u}}$ with the quotient of $\Kims{\Gamma_1} \times_{D^k} \Kimc{\Gamma_2} \times \prod{\mu_{\alpha_i}/\mu_d}$ by the group of automorphisms of the partition $\vec{\alpha}$, that is, the number of bijections $\phi:[k] \rightarrow [k]$ with $\alpha_{\phi(i)} = \alpha_i$. We write $Aut(\Gamma_1,\Gamma_2)$ for the order of this group. We thus have a description of the fixed loci in terms of known stacks. What is not a priori clear is the relation between the natural obstruction theory induced on $F_{\Gamma_1,\Gamma_2}$ from the obstruction theory of $\Kim$ and those of the two factors $\Kims{\Gamma_1},\Kimc{\Gamma_2}$. \\

Let us denote by $\mathcal{T}^1$ and $\mathcal{T}^2$ the tangent and obstruction space for $\Kim$ at a point $f$, as in lemma \ref{lem2}, and by $\mathcal{T}_i^1,\mathcal{T}_i^2$ the tangent and obstruction spaces of the two component maps $f_i$. Denote by $y_1,\cdots,y_k$ the nodes connecting $C_1$ with $C_2$, that is, the nodes over $D=D[0]$. Then $y_1,\cdots,y_k$ become marked points in $C_1,C_2$. 

\begin{lem}\label{lem7}
At the point $f\in F_{\Gamma_1,\Gamma_2,\vec{u}}$ we have the equality 
\begin{align*}
\mathcal{T}^1-\mathcal{T}^2 = \mathcal{T}_1^1+\mathcal{T}_2^1 - \mathcal{T}_1^2-\mathcal{T}_2^2 - (T_D)^k + H^{0}(N_{\mathcal{D}/\mathcal{X}} \tensor N_{\mathcal{D}/\mathcal{Y}})
\end{align*}
in $K$-theory
\end{lem}

\begin{proof} Let $g=\pi \circ f: C \rightarrow W \rightarrow X$ be the composition of $f$ with the contraction $W \rightarrow X$, and let $g_i =\pi \circ f_i$. By lemmas 2 and 3 combined we have that $\mathcal{T}^1,\mathcal{T}^2$ fit into the six term exact sequence

\begin{align*}
0 \rightarrow \mathfrak{aut}(\tau(f)) \rightarrow H^{0} (g^{*}T^{\log}_X) \rightarrow \mathcal{T}^1 \rightarrow \textrm{Def}(\tau(f)) \rightarrow H^{1}(g^{*}T^{\log}_X) \rightarrow \mathcal{T}^2 \rightarrow 0.
\end{align*}

We have two similar six term exact sequences for $f_i$ with $C$ replaced by $C_i$ and $g$ by $g_i$, and $\tau$ replaced by $\tau^{\sim}$ for $f_2$. Write $\vec{x}$ for the vector of marked points on $C$, as above, and $(S_i,y_1, \cdots, y_k)$ for the vector of marked points in $C_i$. Observe that

\begin{align*}
\mathfrak{aut}(C,\vec{x}) = \mathfrak{aut}(C_1,S_1,y_1,\cdots,y_k) \oplus \mathfrak{aut}(C_2,y_1,\cdots,y_k,S_2)
\end{align*}

\noindent -vector fields that vanish on the nodes and marked points of $C$ are simply vector fields that vanish on the nodes and marked points of $C_1$ and $C_2$ and the nodes connecting the two. Furthermore,

\begin{align*}
\mathfrak{aut}_{\mathcal{B}}(W) = \mathfrak{aut}_{\mathcal{B}^{\sim}}(Y)
\end{align*} 

\noindent and thus 

\begin{align*}
\mathfrak{aut}(\tau(f)) = \mathfrak{aut}(\tau(f_1)) \oplus \mathfrak{aut}(\tau^{\sim}(f_2))
\end{align*}

\noindent Similarily, from the local-to-global sequence \ref{(6)} we have an equality in $K$-theory 

\begin{align*}
Ext^{1}(\Omega_{\mathcal{C}}(\sum{x_i}),\mathcal{O}_\mathcal{C}) =& Ext^{1}(\Omega_{\mathcal{C}_1}(\sum{S_1} + y_1 +\cdots y_k),\mathcal{O}_{\mathcal{C}_1})\\ &\bigoplus  Ext^{1}(\Omega_{\mathcal{C}_2}(\sum{y_i} + \sum{S_2}),\mathcal{O}_{\mathcal{C}_2}) \\
&\bigoplus_{i=1}^{k}T_{y_i}\mathcal{C}_1 \tensor T_{y_i}\mathcal{C}_2 
\end{align*}

On the other hand, we have by Corollary \ref{cor1} that $\textrm{Def}(\tau(f))$ differs in $K$-theory from $Ext^{1}((\Omega_{\mathcal{C}}(\sum{x_i}),\mathcal{O}_\mathcal{C})$ by replacing $\bigoplus_{\textup{nodes over D[i]}} T_x\mathcal{C}^1_x \tensor T_x\mathcal{C}^2_x$ with $H^{0}(N_{\mathcal{D}[i]/\mathcal{X}[i-1]} \tensor N_{\mathcal{D}[i]/\mathcal{X}[i]})$. All nodes of $C$ persist as nodes over some target node in $C_1$ and $C_2$, except precisely the nodes over $D$, as $D$ is not a target node for either $f_1$ or $f_2$. Therefore, 

\begin{align*}
Def(C,x) = Def(\tau(f_1)) \oplus Def(\tau^{\sim}(f_2)) \oplus H^{0}(N_{\mathcal{D}/\mathcal{X}} \tensor N_{\mathcal{D}/\mathcal{Y}})
\end{align*} 

\noindent It remains to analyze the cohomology groups $H^{i}(g^{*}T^{\log}_X)$. We have the normalization sequence 

\begin{align*}
0 \rightarrow \mathcal{O}_{C} \rightarrow \mathcal{O}_{C_1} \oplus \mathcal{O}_{C_2} \rightarrow \bigoplus_{i=1}^k \mathcal{O}_{y_i} \rightarrow 0
\end{align*}

\noindent The log tangent bundle $T_X(-\log D)$ fits into the short exact sequence 

\begin{align*}
0 \rightarrow T_X(-\log D) \rightarrow T_X \rightarrow N_{D/X} \rightarrow 0
\end{align*}

\noindent and thus coincides at a point of $D$ with the tangent space at that point in $D$. Therefore, twisting the normalization sequence by $g^{*}T^{\log}_X$ and taking cohomology we get 

\begin{align*}
0 \rightarrow H^{0}(g^{*}T^{\log}_X) \rightarrow H^{0}(g_1^{*}T^{\log}_X) \oplus H^{0}(g_2^{*}T^{\log}_X) \rightarrow (T_D)^k \\
\rightarrow H^{1}(g^{*}T^{\log}_X) \rightarrow H^{1}(g_1^{*}T^{\log}_X) \oplus H^{1} (g_2^{*}T^{\log}_X) \rightarrow 0
\end{align*}

\noindent Therefore, the difference between $H^{0}(g^{*}T^{\log}_X) - H^{1}(g^{*}T^{\log}_X)$ and $\sum H^{0}(g_i^{*}T^{\log}_X) - H^{1}(g_i^{*}T^{\log}_X)$ is precisely $(T_D)^k$. Putting everything together yields the lemma.

\end{proof}

From the lemma it follows that in $K$-theory the vector bundles $\mathcal{T}^{1}-\mathcal{T}^2$ and $\sum \mathcal{T}_i^1 - \mathcal{T}_i^2$ differ by two bundles; the first is the bundle with fiber $(T_D)^k$, which may be identified with the pullback of the tangent bundle $(T_D)^k$ under the evaluation map $\Kims{\Gamma_1} \times_{D^k} \Kimc{\Gamma_2} \rightarrow D^k$; the second one is the bundle with fiber $H^{0}(N_{\mathcal{D}/\mathcal{X}} \tensor N_{\mathcal{D}/\mathcal{Y}})$. This is the line bundle $\mathcal{L}$ that parametrizes deformations of the node $D$; it may be identified with the pullback $p_1^{*}(\mathcal{L}_1) \tensor p_2^{*}(\mathcal{L}_2)$, where $p_1$ and $p_2$ are the projections of $\Kims{\Gamma_1} \times_{D^k} \Kimc{\Gamma_2}$ to the two factors, and $\mathcal{L}_i$ are the respective similar bundles. Note that $\mathcal{L}_1$ is a trivial bundle with nontrivial action, while $\mathcal{L}_2$ is a non-trivial bundle with trivial action. To keep consistent with existing literature, we write $e^T(\mathcal{L}_1)=\frac{w}{d},e^{T}(\mathcal{L}_2)=-\frac{\psi}{d}$. We then obtain: 

\begin{cor}
If $F = F_{\Gamma_1,\Gamma_2,\vec{u}}$ and $N_{\Gamma_1}=(\mathcal{T}_1^1-\mathcal{T}_1^2)^\textup{m}$, we have
\begin{align*}
e^{T}(N_F^{\textup{vir}}) = e^{T}(N_{\Gamma_1})(\frac{w-\psi}{d})
\end{align*}
\end{cor}

\begin{proof} Lemma 7 implies that $N_F^{\textup{vir}} = (\mathcal{T}^1-\mathcal{T}^2)^{\textup{m}}$ differs from the sum of the $(\mathcal{T}_i^1-\mathcal{T}_i^2)^{\textup{m}}$ only by the bundle $\mathcal{L}$, since $T_D$ has trivial action. Furthermore, since the torus action on $\Kimc{\Gamma_2}$ is trivial, the bundles $\mathcal{T}_2^j$ have no moving part.
\end{proof}

From lemma \ref{lem7} it also follows that the map $\Kims{\Gamma_1} \times_{D^k} \Kimc{\Gamma_2} \rightarrow F_{\Gamma_1,\Gamma_2,\vec{u}}$ respects virtual fundamental classes. Let us be more precise. The natural virtual fundamental class of $F_{\Gamma_1,\Gamma_2,\vec{u}}$ is that induced by the fixed part of the perfect obstruction theory for $\Kim$, restricted to $F_{\Gamma_1,\Gamma_2,\vec{u}}$, as explained in the beginning of the section. The natural virtual fundamental class for the fiber product is determined from the diagram

\begin{align*}
\xymatrix{\ar @{} [dr] |{}
\Kims{\Gamma_1} \times_{D^k} \Kimc{\Gamma_2} \ar[d] \ar[r] & \Kims{\Gamma_1} \times \Kimc{\Gamma_2} \ar[d] \\
D^k \ar[r]_v & D^{2k} }
\end{align*}

\noindent as $v^{!}([\Kims{\Gamma_1}]^{\textup{vir}} \times [\Kimc{\Gamma_2}]^{\textup{vir}})$. This is explained in \cite{BF}, section 5. Lemma 5 then implies the corollary: 

\begin{cor}
The virtual fundamental class of $F_{\Gamma_1,\Gamma_2,\vec{u}}$ coincides with the induced virtual fundamental class of the fiber product $\Kims{\Gamma_1} \times_{D^k} \Kimc{\Gamma_2}$. 
\end{cor}

\begin{proof} The condition that the induced perfect obstruction theory $(\mathcal{T}^1-\mathcal{T}^2)^{\textup{f}}$ differs from the sum of the perfect obstruction theories $(\mathcal{T}_i^1-\mathcal{T}_i^2)^{\textup{f}}$ of the factors by the pullback of $(T_D)^{k} = N_{D^k/D^{2k}}$ under the natural map $\Kims{\Gamma_1} \times_{D^k} \Kimc{\Gamma_2}$ is precisely the compatibility condition on perfect obstruction theories required by section 5 of \cite{BF}, which ensures that the induced virtual fundamental classes coincide. 
\end{proof}

Putting everything together, we obtain 

\begin{thm}{Log Virtual Localization}:
\begin{align*}
[\Kim]^{\textup{vir}} =& [\Kims{\Gamma}]^{\textup{vir}} + \sum_{\Gamma_1,\Gamma_2,\vec{u}} \frac{[\Kims{\Gamma_1} \times_{D^k} \Kimc{\Gamma_2}]^{\textup{vir}}}{Aut(\Gamma_1,\Gamma_2)(\frac{w-\psi}{d})e(N_{\Gamma_1})} = \\
&[\Kims{\Gamma}]^{\textup{vir}} + \sum_{\Gamma_1,\Gamma_2} \frac{\prod{\alpha_i}}{d} \frac{[\Kims{\Gamma_1} \times_{D^k} \Kimc{\Gamma_2}]^{\textup{vir}}}{Aut(\Gamma_1,\Gamma_2)(\frac{w-\psi}{d})e(N_{\Gamma_1})}
\end{align*}
\label{thm:2}
\end{thm}

\noindent In the last equality we have replaced the choice of the units $\vec{u}$ with their number $\frac{\prod{\alpha_i}}{d}$. \\

In section 2, we discussed the finite map $\pi: \Kim \rightarrow \Li$ from the moduli space of log stable maps to the moduli space of relative stable maps. In the paper \cite{AMW} it is shown that the pushforward of the virtual fundamental class of $\Kim$ under $\pi$ coincides with the virtual fundamental class of Jun Li's space. We may modify these results to include the maps $\Kims{\Gamma_1} \rightarrow \mathcal{M}_{\Gamma_1}^{\textup{sim}}(X,D)=\mathcal{M}_{\Gamma_1}^{\textup{sim}}$ and $\Kimc{\Gamma_2} \rightarrow \mathcal{M}_{\Gamma_2}^{\sim}$, with the appropriate modifications of the spaces in the setting of relative stable maps as targets. Then, applying $\pi_{*}$ to both sides of the equation in theorem 2 yields the relative virtual localization theorem of Graber-Vakil:

\begin{cor}
The log virtual localization formula becomes the relative virtual localization formula under the functor $\pi_{*}$. 
\end{cor}

\begin{proof}
By the results of \cite{AMW}, 
\begin{align*}
\pi_{*}[\Kims{\Gamma_1} \times_{D^k} \Kimc{\Gamma_2}]^{\textup{vir}} = [\mathcal{M}_{\Gamma_1}^{\textup{sim}} \times_{D^k} \mathcal{M}_{\Gamma_2}^{\sim}]^{\textup{vir}}
\end{align*}

\noindent  So what remains is to analyze the Euler classes appearing in the denominators of the formula. The term $e(N_\Gamma)$ does not change under $\pi$, as it is the Euler class of the virtual normal bundle of a fixed locus inside the simple locus, and $\pi$ is an isomorphism over the simple locus. On the other hand, let $L$ be the line bundle in $\Li$ parametrizing deformations of the node; its fiber at a point is $H^{0}(C,N_{D/X} \tensor N_{D/Y})$. We denote the Euler class of $L$ by $w-\psi$, as in \cite{GV}. Note that over a fixed locus $F_{\Gamma_1,\Gamma_2,\vec{u}}$, the pullback $\pi^{*}(L) \cong \mathcal{L}^{d}$, where $\mathcal{L}$ is the line bundle of $\Kim$ with fiber $H^{0}(N_{\mathcal{D}/\mathcal{X}} \tensor N_{\mathcal{D}/\mathcal{Y}})$ parametrizing deformations of the node $\mathcal{D}$, which is the $d$-th root of $D$. We thus have $e(\pi^{*}L) = de(\mathcal{L}) = d(\frac{w-\psi}{d})$, which justifies the choice of notation for $e(\mathcal{L})$. The push-pull formula then yields

\begin{align*}
\pi_{*} (\frac{\prod{\alpha_i}}{d} \frac{[\Kims{\Gamma_1} \times_{D^k} \Kimc{\Gamma_2}]^{\textup{vir}}}{Aut(\Gamma_1,\Gamma_2)(\frac{w-\psi}{d})e(N_{\Gamma_1})}) = \prod{\alpha_i} \frac{[\mathcal{M}_{\Gamma_1}^{\textup{sim}} \times_{D^k} \mathcal{M}_{\Gamma_2}^{\sim}]^{\textup{vir}}}{Aut(\Gamma_1,\Gamma_2)(w-\psi)e(N_{\Gamma_1})}
\end{align*}

\noindent Summing over all $\Gamma_1,\Gamma_2$ gives precisely the relative virtual localization formula of \cite{GV}.
\end{proof}

\end{subsection}

\end{section}
\bibliography{refs}{}
\end{document}